\documentclass[a4paper]{amsart}

\RequirePackage{amsmath} \RequirePackage{amssymb}
\usepackage{amscd,latexsym,amsthm,amsfonts,amssymb,amsmath,amsxtra}
\usepackage[colorlinks=true,urlcolor=blue,citecolor=blue]{hyperref}
\usepackage{color}
\usepackage[all]{xy}
\usepackage[OT2,T1]{fontenc}
\usepackage{bm}
\DeclareSymbolFont{cyrletters}{OT2}{wncyr}{m}{n}
\DeclareMathSymbol{\Sha}{\mathalpha}{cyrletters}{"58}

\let\Re\undefined
\let\Im\undefined

\DeclareMathOperator{\Re}{Re}
\DeclareMathOperator{\Im}{Im}
\DeclareMathOperator{\Gal}{Gal}

\DeclareMathOperator{\vol}{vol}

\begin{document}

    \theoremstyle{plain}
    \newtheorem{thm}{Theorem} \newtheorem{cor}[thm]{Corollary}
    \newtheorem{thmy}{Theorem}
    \renewcommand{\thethmy}{\Alph{thmy}}
    \newenvironment{thmx}{\stepcounter{thm}\begin{thmy}}{\end{thmy}}
    \newtheorem{lemma}[thm]{Lemma}  \newtheorem{prop}[thm]{Proposition}
    \newtheorem{conj}[thm]{Conjecture}  \newtheorem{fact}[thm]{Fact}
    \newtheorem{claim}[thm]{Claim}
    \theoremstyle{definition}
    \newtheorem{defn}[thm]{Definition}
    \newtheorem{example}[thm]{Example}
    \newtheorem{exercise}[thm]{Exercise}
    \theoremstyle{remark}
    \newtheorem*{remark}{Remark}

    \newcommand{\BA}{{\mathbb {A}}} \newcommand{\BB}{{\mathbb {B}}}
    \newcommand{\BC}{{\mathbb {C}}} \newcommand{\BD}{{\mathbb {D}}}
    \newcommand{\BE}{{\mathbb {E}}} \newcommand{\BF}{{\mathbb {F}}}
    \newcommand{\BG}{{\mathbb {G}}} \newcommand{\BH}{{\mathbb {H}}}
    \newcommand{\BI}{{\mathbb {I}}} \newcommand{\BJ}{{\mathbb {J}}}
    \newcommand{\BK}{{\mathbb {K}}} \newcommand{\BL}{{\mathbb {L}}}
    \newcommand{\BM}{{\mathbb {M}}} \newcommand{\BN}{{\mathbb {N}}}
    \newcommand{\BO}{{\mathbb {O}}} \newcommand{\BP}{{\mathbb {P}}}
    \newcommand{\BQ}{{\mathbb {Q}}} \newcommand{\BR}{{\mathbb {R}}}
    \newcommand{\BS}{{\mathbb {S}}} \newcommand{\BT}{{\mathbb {T}}}
    \newcommand{\BU}{{\mathbb {U}}} \newcommand{\BV}{{\mathbb {V}}}
    \newcommand{\BW}{{\mathbb {W}}} \newcommand{\BX}{{\mathbb {X}}}
    \newcommand{\BY}{{\mathbb {Y}}} \newcommand{\BZ}{{\mathbb {Z}}}

    \newcommand{\CA}{{\mathcal {A}}} \newcommand{\CB}{{\mathcal {B}}}
    \newcommand{\CC}{{\mathcal {C}}} \renewcommand{\CD}{{\mathcal {D}}}
    \newcommand{\CE}{{\mathcal {E}}} \newcommand{\CF}{{\mathcal {F}}}
    \newcommand{\CG}{{\mathcal {G}}} \newcommand{\CH}{{\mathcal {H}}}
    \newcommand{\CI}{{\mathcal {I}}} \newcommand{\CJ}{{\mathcal {J}}}
    \newcommand{\CK}{{\mathcal {K}}} \newcommand{\CL}{{\mathcal {L}}}
    \newcommand{\CM}{{\mathcal {M}}} \newcommand{\CN}{{\mathcal {N}}}
    \newcommand{\CO}{{\mathcal {O}}} \newcommand{\CP}{{\mathcal {P}}}
    \newcommand{\CQ}{{\mathcal {Q}}} \newcommand{\CR}{{\mathcal {R}}}
    \newcommand{\CS}{{\mathcal {S}}} \newcommand{\CT}{{\mathcal {T}}}
    \newcommand{\CU}{{\mathcal {U}}} \newcommand{\CV}{{\mathcal {V}}}
    \newcommand{\CW}{{\mathcal {W}}} \newcommand{\CX}{{\mathcal {X}}}
    \newcommand{\CY}{{\mathcal {Y}}} \newcommand{\CZ}{{\mathcal {Z}}}

    \newcommand{\RA}{{\mathrm {A}}} \newcommand{\RB}{{\mathrm {B}}}
    \newcommand{\RC}{{\mathrm {C}}} \newcommand{\RD}{{\mathrm {D}}}
    \newcommand{\RE}{{\mathrm {E}}} \newcommand{\RF}{{\mathrm {F}}}
    \newcommand{\RG}{{\mathrm {G}}} \newcommand{\RH}{{\mathrm {H}}}
    \newcommand{\RI}{{\mathrm {I}}} \newcommand{\RJ}{{\mathrm {J}}}
    \newcommand{\RK}{{\mathrm {K}}} \newcommand{\RL}{{\mathrm {L}}}
    \newcommand{\RM}{{\mathrm {M}}} \newcommand{\RN}{{\mathrm {N}}}
    \newcommand{\RO}{{\mathrm {O}}} \newcommand{\RP}{{\mathrm {P}}}
    \newcommand{\RQ}{{\mathrm {Q}}} \newcommand{\RR}{{\mathrm {R}}}
    \newcommand{\RS}{{\mathrm {S}}} \newcommand{\RT}{{\mathrm {T}}}
    \newcommand{\RU}{{\mathrm {U}}} \newcommand{\RV}{{\mathrm {V}}}
    \newcommand{\RW}{{\mathrm {W}}} \newcommand{\RX}{{\mathrm {X}}}
    \newcommand{\RY}{{\mathrm {Y}}} \newcommand{\RZ}{{\mathrm {Z}}}

    \newcommand{\fa}{{\mathfrak{a}}} \newcommand{\fb}{{\mathfrak{b}}}
    \newcommand{\fc}{{\mathfrak{c}}} \newcommand{\fd}{{\mathfrak{d}}}
    \newcommand{\fe}{{\mathfrak{e}}} \newcommand{\ff}{{\mathfrak{f}}}
    \newcommand{\fg}{{\mathfrak{g}}} \newcommand{\fh}{{\mathfrak{h}}}
    \newcommand{\fii}{{\mathfrak{i}}} \newcommand{\fj}{{\mathfrak{j}}}
    \newcommand{\fk}{{\mathfrak{k}}} \newcommand{\fl}{{\mathfrak{l}}}
    \newcommand{\fm}{{\mathfrak{m}}} \newcommand{\fn}{{\mathfrak{n}}}
    \newcommand{\fo}{{\mathfrak{o}}} \newcommand{\fp}{{\mathfrak{p}}}
    \newcommand{\fq}{{\mathfrak{q}}} \newcommand{\fr}{{\mathfrak{r}}}
    \newcommand{\fs}{{\mathfrak{s}}} \newcommand{\ft}{{\mathfrak{t}}}
    \newcommand{\fu}{{\mathfrak{u}}} \newcommand{\fv}{{\mathfrak{v}}}
    \newcommand{\fw}{{\mathfrak{w}}} \newcommand{\fx}{{\mathfrak{x}}}
    \newcommand{\fy}{{\mathfrak{y}}} \newcommand{\fz}{{\mathfrak{z}}}
     \newcommand{\fA}{{\mathfrak{A}}} \newcommand{\fB}{{\mathfrak{B}}}
    \newcommand{\fC}{{\mathfrak{C}}} \newcommand{\fD}{{\mathfrak{D}}}
    \newcommand{\fE}{{\mathfrak{E}}} \newcommand{\fF}{{\mathfrak{F}}}
    \newcommand{\fG}{{\mathfrak{G}}} \newcommand{\fH}{{\mathfrak{H}}}
    \newcommand{\fI}{{\mathfrak{I}}} \newcommand{\fJ}{{\mathfrak{J}}}
    \newcommand{\fK}{{\mathfrak{K}}} \newcommand{\fL}{{\mathfrak{L}}}
    \newcommand{\fM}{{\mathfrak{M}}} \newcommand{\fN}{{\mathfrak{N}}}
    \newcommand{\fO}{{\mathfrak{O}}} \newcommand{\fP}{{\mathfrak{P}}}
    \newcommand{\fQ}{{\mathfrak{Q}}} \newcommand{\fR}{{\mathfrak{R}}}
    \newcommand{\fS}{{\mathfrak{S}}} \newcommand{\fT}{{\mathfrak{T}}}
    \newcommand{\fU}{{\mathfrak{U}}} \newcommand{\fV}{{\mathfrak{V}}}
    \newcommand{\fW}{{\mathfrak{W}}} \newcommand{\fX}{{\mathfrak{X}}}
    \newcommand{\fY}{{\mathfrak{Y}}} \newcommand{\fZ}{{\mathfrak{Z}}}
    \newcommand{\Supp}{\operatorname{Supp}}
    \newcommand{\coker}{\operatorname{coker}}
\title[Explicit CM Type Norm Formula and Effective Nonvanishing]{An Explicit CM Type Norm Formula and Effective Nonvanishing of Class Group L-functions for CM Fields}%
\author[Liyang Yang]{ Liyang Yang}

\address{253-37 Caltech, Pasadena\\
CA 91125, USA}
\email{lyyang@caltech.edu}

\begin{abstract}
We show that the central value of class group $L$-functions of general CM fields can be expressed in terms of derivatives of real-analytic Hilbert Eisenstein series at CM points. Using this in conjunction with an explicit CM type norm formula established in Section \ref{sec3}, following an idea of Iwaniec and Kowalski \cite{IK04}, we obtain a conditional explicit lower bound for class numbers of CM fields under the assumption $\zeta_K(1/2)\ll_F
\log D_{K/F}$ (note that GRH implies $\zeta_K(1/2)\leq0$). Some results in the proof lead to an \textit{effective} nonvanishing result for class group L-functions of general CM fields, generalizing the only known ineffective results. Moreover, combining the CM type norm formula with Barquero-Sanchez and Masri's work in \cite{BM16}, one shall deduce an explicit Chowla-Selberg formula for \textit{all} abelian CM fields.
\end{abstract}

\maketitle

\section{Introduction}
\subsection{A Lower Bound for the Class Number of CM Fields}
For imaginary quadratic fields $K=\mathbb{Q}(\sqrt{-D}),$ Gauss' class number problem has for a long time inspired the study of lower bounds of $h(-D),$ the class number of $K.$ Also, the magnitude of $h(-D)$ is closely related to the exceptional characters, i.e., those characters $\chi$ such that $L(s,\chi)$ has a real zero near $s=1$ (ref. \cite{Lan18}, \cite{Go75}, \cite{GS75} and \cite{GS00}). A repelling property of the exceptional zero gives the result $h(-D)\rightarrow \infty$ as $D\rightarrow \infty$ (ref. \cite{D33} and \cite{H34}). Landau then obtained the lower bound $h(-D)\gg_{\epsilon} D^{1/8-\epsilon}$ (ref. \cite{Lan36}) by a quantitative analysis of the repelling effects. In the same year, Siegel got a stronger result: $h(-D)\gg_{\epsilon} D^{1/2-\epsilon}$ (ref. \cite{Si36}). See \cite{Iw06} for a more concrete historical introduction. However, all these results suffer from the serious defect of being ineffective. Hence one can not use them to determine the fields of class number one. Also, there are many other situations requiring an effective lower bound for $h(-D),$ for example, by genus theory, the Euler idoneal number problem calls for an effective lower bound 
\begin{align*}
h(-D)\gg D^{c'/\log\log D},\ \text{with $c'>\log 2.$}
\end{align*}
Generally one hopes to show that $h_{K}\geq C D_K^c$ for some positive (absolute) constant $c$ and $C,$ where $h_K$ is the class number and $D_K$ is the absolute discriminant. 

Unconditionally, Stark (ref. Theorem 2 in \cite{Str74}) gave an effective lower bound for the class number of a CM field $K$ of the shape $h_K\gg_{n,\epsilon} D_{K/F}^{1/2-1/2n}D_F^{\epsilon},$ where $\epsilon>0,$ $F$ is a totally real subfield of $K$ with $n=[F:\mathbb{Q}],$ and $D_F$ is the absolute discriminant of $F$ and $D_{K/F}=D_K\cdot D_F^{-2}.$ When $[K:\mathbb{Q}]\geq 20,$ Hoffstein (ref. \cite{Ho79}) generalized Stark's result by removing the $D_F^{\epsilon}$-term and computing the implied constant explicitly. Also, for any CM field $K,$ Odlyzko (ref.  \cite{O75}) gives an effective lower bound of $h_K$ in terms of the degree of $K.$ For $n=1$ the known unconditional lower bound is $\prod_{p\mid D_K}\left(1-{2\sqrt{p}}/({p+1})\right)^{-1}\log D_K$ due to Goldfeld, Gross and Zagier (ref. \cite{Go85} and \cite{GZ86}). On the other hand, if we assume the Grand Riemann Hypothesis, the exponent $c$ can be taken to be $1/2-\epsilon$ for any $\epsilon>0.$ Moreover, assuming the Dedekind zeta function $\zeta_K(s)$ has a Siegel zero, Louboutin (ref. \cite{Lou94}) obtained effective lower bounds for $h_K,$ with $c=1/4.$

It is well known that the class group $L$-functions of an imaginary quadratic field $K=\mathbb{Q}(\sqrt{-D})$ can be expressed in terms of values of the real-analytic Eisenstein series for $SL_2(\mathbb{Z})$ at Heegner points (ref. \cite{DFI95}). Based on this fact, Iwaniec and Kowalski \cite{IK04} obtained an effective lower bound for the class number $h_K\gg D^{1/4}\log D$ assuming that $L(1/2,\chi_{K/{\mathbb{Q}}})\geq0,$ where $\chi_{K/{\mathbb{Q}}}$ is the quadratic character corresponding the extension $K/{\mathbb{Q}}$ (ref. Section \ref{sec2}). In this paper we generalize Iwaniec and Kowalski's result to arbitrary CM fields and obtain an expression for class group $L$-functions in terms of derivatives of real-analytic Hilbert Eisenstein series at CM points. Due to the estimates on the Fourier expansions we show that for any CM field $K,$ if $\zeta_K(1/2)\leq0,$ then $c=1/4$ is admissible, and the implied constant is effective.

A precise statement is our Theorem \ref{main}. To achieve it, we shall introduce some analytic objects with respect to $F.$ We let $\gamma_F^*$ be the normalized Euler-Kronecker constant, i.e.,
\begin{equation}\label{normized}
\gamma_F^*=\lim_{s\rightarrow1}\left(\rho_F^{-1}\zeta_F(s)-\frac{1}{s-1}\right).
\end{equation}
For any CM extension $K/F,$ we will always denote by $D_F$ (resp. $D_K$) the absolute discriminant of $F$ (resp. $K$). Let $h_F$ (resp. $h_K$) be the class number of $F$ (resp. $K$). Then we have
\begin{thmx}\label{main}
Let $F/\mathbb{Q}$ be a totally real field of degree n. Let $K/F$ be a CM extension. Assume that
\begin{equation}\label{assump}
\zeta_K(1/2)\leq \frac{\rho_F}{4[\mathcal{O}^{\times}_K:\mathcal{O}^{\times}_F]\cdot h_F}
\log\frac{\sqrt{D_K}}{D_F}.
\end{equation}
Then we have
\begin{equation}\label{0.1}
h_K\geq{\Psi_F}^{-1}\cdot D_K^{1/4}\log\frac{\sqrt{D_K}}{D_F},
\end{equation}
where
$$
\Psi_F:=\pi^{-n}\rho_F^{-1}h_FD_F^{2}+e^{2\gamma_F^*+2n+2}\Upsilon_{F}^*D_F^{3/2} ,
$$
and $\Upsilon_{F}^*=4\gamma_F^*+3\log D_F+4n+\sqrt{7n}+8.$
\end{thmx}

\begin{remark}
	Noting that $\zeta_K(s)$ is continuous on $(1/2,1)$ and $\lim_{s\rightarrow 1^-}\zeta_K(s)=-\infty,$ the Grand Riemann Hypothesis then gives $\zeta_K(1/2)\leq0,$ which is stronger than the assumption \eqref{assump}. Unconditionally, people only know that $\zeta_{\mathbb{Q}}(1/2)<0$ currently. Also, we can keep the normalized Euler-Kronecker constant $\gamma_F^*$ here as an invariant of $F$. One can refer Theorem 7 in \cite{MJ07} for an elementary upper bound for $\gamma_F^*$ and Theorem 1 in \cite{Ih06} for a conditional upper bound.
\end{remark}
\begin{remark}
Note that Stark's bound is sharper than \eqref{0.1} except $[K:\mathbb{Q}]\leq 4.$ However, the proof of Theorem \ref{main}, when invoking with Stark's result, leads to effective results on nonvanishing of class group L-functions (and their derivatives) for CM fields. Note that previous results are all ineffective. See Section \ref{2} for more details.

Also, we point out that, with a little bit more work, the inequality \eqref{0.1} can be naturally generalized to a conditional lower bound for $h_{\mathfrak{O}},$ where $\mathfrak{O}$ is an order of $K$ and $h_{\mathfrak{O}}$ denotes the number of proper $\mathfrak{O}-$ideal classes of $\mathfrak{O},$ since everything in this paper has a counterpart in the order case. 
\end{remark}

The outline of the proof to Theorem \ref{main} is described in Section \ref{sec2.2} below, with the new ingredients involved. See the rest of the sections for lengthy details.

Also, we have the following corollary to go beyond \cite{IK04} by plugging the upper bound for $\gamma_F^*$ given in Lemma \ref{28''} into \eqref{0.1}. It is of course weaker than the conditional result in \cite{Ih06} under GRH; however, it is simple enough compared to the elementary bound in \cite{MJ07}. 
\begin{cor}\label{cormain}
	Let notation be as before. Then there exist absolute constants $c_1,$ $c_2>0$ such that if $\zeta_K(1/2)$ satisfies \eqref{assump}, then we have
	\begin{equation}\label{0.01}
	h_K^-\geq c_1D_F^{-c_2}\cdot D_K^{1/4}\log D_{K/F},
	\end{equation}
	where, $h_K^{-}$ is the relative class number of $K/F$ and $D_{K/F}=D_F^{-2}D_K$ is the relative discriminant of $K/F.$ 
	\end{cor} 
\begin{remark}
Combining the analytic class number formula with a result of Louboutin (ref. Proposition A of \cite{Lou94}), one actually sees that if we assume $\zeta_K(1/2)\leq0,$ then essentially $h_K^-\gg D_K^{1/4},$ where the implied constant is effective. According to the conditional estimate $\gamma_F^*\ll \log\log D_F$ (ref.  \cite{Ih06}), it is expected that one can take $c_2=3+\varepsilon$ in \eqref{0.01} for any $\varepsilon>0.$ So \eqref{0.01} is an improvement of Louboutin's result if we fix $F$ and let $D_K$ vary. Also, by Hermite's theorem (ref. \cite{He57}) one sees that almost all $h_K^{-}$ is a positive power of $D_{K/F}.$
\end{remark}

\subsection{Nonvanishing of Class Group $L$-functions}\label{2}
It is an important problem in Number Theory to determine whether an interesting $L$-function is nonvanishing at the central value $s=1/2.$ However, it is usually pretty difficult to check individually. A common strategy is to consider instead the average of $L$-functions over a family. One fruitful way for dealing with such averages combines periods relations of Waldspurger type with the equidistribution of special points on varieties (see e.g. \cite{MV05}, \cite{Ma10}). However, in this paper, we will use purely analytic methods to obtain asymptotic expressions for some averages of weighted class group $L$-functions. This way has several advantages. For instance, one can obtain effective nonvanishing results, namely, avoiding using Siegel's bound.

Although one can obtain the nonvanishing results on the critical line $\Re(s)=1/2$ by the same way, we just focus ourselves here on the central value $s=1/2$ for simplicity. Precisely, we will use byproducts from the proof of Theorem \ref{main} to obtain a lower bound for the number of nonvanishing class group L-functions as \eqref{cor*}. This bound, as can be seen, is almost as good as the conjectural magnitude. Moreover, a significant difference between our approach and that in \cite{Ma10} is that subconvexity bounds are not essential for us. Therefore, one might be able to handle higher derivative case by general convex bounds for Dirichlet series via our methods here. 
\begin{thmx}\label{c}
	Let $F/\mathbb{Q}$ be a totally real field and $K/F$ is a CM extension. Denote by $L^{(0)}_K(\chi,1/2)=L_K(\chi,1/2)$ and $L^{(1)}_K(\chi,1/2)=L'_K(\chi,1/2),$ the derivative of $L_K(\chi,s)$ at $s=1/2.$ Then for any $\varepsilon>0$ we have
	\begin{equation}\label{cor*}
	\#\Big\{\chi\in\widehat{Cl(K)}:\; L^{(k)}_K\left(\chi,1/2\right)\neq0\Big\}\gg_{F} \frac{h_K}{\log D_K},\quad k=0,1,
	\end{equation}
	where the implied constant in \eqref{cor*} is computable.
\end{thmx}
\begin{remark}
When $F=\mathbb{Q},$ i.e. $K$ is imaginary quadratic, Blomer followed \cite{DFI95} to prove a better lower bound (ref. \cite{B04})
	$$
	\#\Big\{\chi\in\widehat{Cl(K)}:\; L_K\left(\chi,1/2\right)\neq0\Big\}\geq c\cdot h_K\prod_{p\mid D_K}\left(1-\frac1p\right)
	$$
	for some constant $c>0$ and all sufficiently large $D_K.$ Once $D_K$ is chosen, the constant $c$ can be taken explicitly. However, one does not know how large $D_K$ must be chosen for this lower bound to be valid because of an application of Siegel's lower bound for $L(1,\chi_{K/F}).$ For a general CM extension $K/F$, when $F$ has trivial class group, Masri (ref. \cite{Ma10}) proved a lower bound as  
	\begin{equation}\label{M}
	\#\Big\{\chi\in\widehat{Cl(K)}:\; L_K\left(\chi,1/2\right)\neq0\Big\}\gg_{F,\varepsilon} D_K^{1/100-\varepsilon}.
	\end{equation}
	 Again, this bound is ineffective since Siegel's bound is used here. The exponent $1/100$ comes from an application of  subconvexity bound for $GL(2).$
\end{remark}

Together with Stark's effective lower bound, Theorem \ref{c} above implies that 
\begin{cor}\label{cor*1}
	Let $F/\mathbb{Q}$ be a totally real field of degree $n$ and $K/F$ a CM extension. Then we have
	\begin{equation}\label{cor**}
	\#\Big\{\chi\in\widehat{Cl(K)}:\; L^{(k)}_K\left(\chi,1/2\right)\neq0\Big\}\gg_{F,\varepsilon} \frac{D_K^{\frac{1}{2}-\frac1{2n}}}{\log D_K},\quad k=0,1,
	\end{equation}
Moreover, the implied constants in \eqref{cor**} are computable. 
\end{cor}

\begin{remark}
	
	By \eqref{cor*} and Theorem \ref{main}, if we assume \eqref{assump}, in particular, if we assume $\zeta_K(1/2)\leq0$, then one obtains
	\begin{equation}
	\#\Big\{\chi\in\widehat{Cl(K)}:\; L^{(k)}_K\left(\chi,1/2\right)\neq0\Big\}\gg_{F,\varepsilon} D_K^{1/4},\quad k=0,1,
	\end{equation}
	where the implied constant is computable. Moreover, Siegel's theorem gives that
	\begin{equation}\label{15}
	\#\Big\{\chi\in\widehat{Cl(K)}:\; L^{(k)}_K\left(\chi,1/2\right)\neq0\Big\}\gg_{F,\varepsilon} \frac{D_K^{1/2}}{\log D_K},\quad k=0,1,
	\end{equation}
	where the implied constant is ineffective. Note that \eqref{15} is a significant improvement of \eqref{M}.
\end{remark}

\subsection{An Explicit Chowla-Selberg Formula for General Abelian CM Fields}
The celebrating Chowla-Selberg formula was first proved by Chowla and Selberg for imaginary quadratic fields in \cite{CS67} by analytic methods. A geometric interpretation was given by Gross (ref. \cite{G78}).
Yoshida obtained such a formula for arbitrary CM fields in \cite{Yo03}.  For any abelian CM field which contains a totally real subfield with trivial narrow class group, Barquero-Sanchez and Masri (ref. \cite{BM16}), combining with Lerch's identity (ref. \cite{L97}) and the results in \cite{D84}, were able to obtain an explicit Chowla-Selberg formula in terms of (generalized) gamma functions, paralleling the original Chowla-Selberg formula.

In this section, we point out that Barquero-Sanchez and Masri's idea works for $all$ CM fields with our Proposition \ref{1} being the new input. Hence we shall avoid repeating the proof of \cite{BM16} and just state the generalized formula below.

Let $K$ be an abelian CM field and $F$ is its maximal totally real subfield. Then there is some $N\in\mathbb{N}_{\geq1}$ such that $K\subset \mathbb{Q}(\zeta_N)$ with $\zeta_N=e^{2\pi i/N}.$ Let $H_K$ (resp. $H_F$) be the subgroup of $\Gal(\mathbb{Q}(\zeta_N)/\mathbb{Q})$ which fixes $K$ (resp. $F$). Fix an isomorphism $\Gal(\mathbb{Q}(\zeta_N)/\mathbb{Q})\simeq \left(\mathbb{Z}/N\mathbb{Z}\right) ^{\times}.$ Set $X_K=\{\chi\in \widehat{\left(\mathbb{Z}/N\mathbb{Z}\right) ^{\times}}:\ \chi\mid_{X_K}=1\}.$ Let $X_F$ be defined similarly. For any $\chi\in X_K,$ it can be written as $\chi=\chi^{*}\chi_0,$ where $\chi^*$ is primitive and $\chi_0$ is trivial. Write $c_{\chi}$ the conductor of $\chi^*.$ Since $\chi^{*}$ is uniquely determined by $\chi,$ then $c_{\chi}$ is well-defined. Define the Gauss sum associated with $\chi\in X_E$ to be
$$
\tau(\chi)=\sum_{k=1}^{c_{\chi}}\chi(k)e^{2\pi ki/c_{\chi}}.
$$
Define the function $\Gamma_2(w)=e^{R(w)},$ $\Re(w)>0,$ where
$$
R(x)=\lim_{n\rightarrow \infty}\left(-\zeta_{\mathbb{Q}}''(0)+x\log^2n-\log^2x-\sum_{k=1}^{n-1}\left(\log^2(x+k)-\log^2k\right) \right).
$$
This function $\Gamma_2(w)$ is defined in \cite{D84} and it is analogous to $\Gamma(s)/\sqrt{2\pi}.$
Now we can state a general version of the Chowla-Selberg formula for abelian CM fields.

\begin{thmx}\label{9}
Let $F/\mathbb{Q}$ be a totally real field of degree $n$ and $K/F$ a CM extension with $K/\mathbb{Q}$ abelian. Let $\Phi$ be the fixed CM type for $K$ as before. For any fractional ideal $\mathfrak{a}$ of $K,$ denote by $\mathfrak{f}_{\mathfrak{a}}$ a fixed integral ideal in the Steinitz class of $\mathfrak{a}$ with minimal absolute norm, and write $z_{\mathfrak{a}}$ for the corresponding CM point of type $(K,\Phi).$ Then one has
\begin{align*}
\prod_{[\mathfrak{a}]\in Cl(K)}H(z_{\mathfrak{a}},\mathfrak{f}_{\mathfrak{a}})=C_{F}^K\prod_{\chi\in X_K\setminus X_F}\prod_{k=1}^{c_{\chi}}\Gamma\left(\frac{k}{c_{\chi}}\right) ^{\frac{h_K\chi(k)}{2L(0,\chi)}}\prod_{\substack{\chi\in X_F\\ \chi\neq 1}}\prod_{k=1}^{c_{\chi}}\Gamma_2\left(\frac{k}{c_{\chi}}\right) ^{\frac{h_K\tau(\chi)\overline{\chi}(k)}{2c_{\chi}L(1,\chi)}},
\end{align*}
where for any fractional ideal $\mathfrak{b}$ of $F,$ $H(z;\mathfrak{b})=\big[N_{\Phi}(\Im (z)){N_{F/\mathbb{Q}}(\mathfrak{b})}^{-1}\big]^{1/h_F}\phi(z;\mathfrak{b}),$  and $\phi(z;\mathfrak{b})$ is defined via
\begin{align*}
\log\phi(z;\mathfrak{b})&=-\frac{2\pi^nD^{1/2}_Fy^{\bm{\sigma}}}{\rho_FN_{F/\mathbb{Q}}(\mathfrak{b})}\cdot\zeta_F(2,[\mathcal{O}_K])-\frac{2\rho_F^{-1}}{D_FN_{F/\mathbb{Q}}(\mathfrak{b})}\sum_{b\in F^{\times}}|N(b)|^{-1}\lambda(b,0)\mathbf{e}(bz),
\end{align*}
and 
$$
\lambda(b,s)=\mathop{\sum\sum}_{\substack{(a,c)\in\mathfrak{b}^{-1}\mathfrak{o}^{-1}\times\mathcal{O}_F^{\times}/\mathcal{O}^{\times,+}_F\\ac=b}}|c|^{(1-2s)\bm{\sigma}}.
$$ 
Also, the constant $C_{F}^K=\left({2^{-n-1}\pi^{-1}{D_K}^{-1/2}}{D_F}\right) ^{h_K/2}.$
\end{thmx}

\begin{remark}
It is clear that by definition $H(z_{\mathfrak{a}},\mathfrak{f}_{\mathfrak{a}})$ is independent of the choice of $\mathfrak{f}_{\mathfrak{a}}.$ If $h_F^-=1,$ we may take $\mathfrak{f}_{\mathfrak{a}}=\mathcal{O}_F$ to recover Theorem 1.1 in \cite{BM16}. Combining with Colmez's theorem (ref. \cite{Col93}), the formula above can be used to compute the average of Faltings heights of certain CM abelian varieties. 
\end{remark}

\textbf{Acknowledgements.}
I am grateful to my advisor Dinakar Ramakrishnan for very helpful discussions and comments. I would like to thank Tonghai Yang for his careful reading and suggestions; and Jeffrey Hoffstein, St\'ephane  Louboutin for their comments. Sincere thanks are also due to Zavosh Amir-Khosravi for his kind help.

\section{Proof of Theorem \ref{main} in Imaginary Quadratic Case}\label{sec2}
\subsection{Review of the Imaginary Quadratic Case}
We start by reviewing the case of Theorem \ref{main} in the imaginary quadratic case. This is Iwaniec-Kowalski's original idea. For the sake of illustration, we give a brief proof following \cite{IK04}.

Let $K=\mathbb{Q}\left(\sqrt{-D}\right)$ be a imaginary quadratic field. Since $\mathbb{Q}$ has class number 1, we can often factor a non-zero integral ideal uniquely as $(l)\mathfrak{a}$ where $l\in\mathbb{Z}_{>0}$ and $\mathfrak{a}$ is a primitive ideal, i.e., $\mathfrak{a}$ has no rational integer factors other than $\pm1$.

If $\mathfrak{a}$ is primitive, then it is generated by $\mathfrak{a}=\Big[a,\frac{b+i\sqrt{D}}2\Big],$ where $a=N\mathfrak{a}$ and $b$ solves the congruence $b^2+D\equiv0\ (\; mod\; 4a),$ and is determined modulo $2a.$

Conversely, given such $a$ and $b$ we get a primitive ideal $\mathfrak{a}=\Big[a,\frac{b+i\sqrt{D}}2\Big].$ Thus there exists a one-to-one correspondence between the primitive ideals and the points
$$
z_{\mathfrak{a}}:=\frac{b+i\sqrt{D}}{2a}\in \mathbb{H}\quad\text{determined by modulo 1}.
$$
These will be called the Heegner points. Moreover, we have $\mathfrak{a}^{-1}=[1,\bar{z}_{\mathfrak{a}}].$ Then according to \cite{DFI95} one has the following formula:
\begin{equation}\label{1.1}
\frac1h\sum_{\chi\in\widehat{Cl}(K)}\chi(\mathfrak{a})L_K\left(s,\chi\right)=w^{-1}\left(\frac{\sqrt{D}}2\right)^{-s}\zeta(2s)E\left(z_{\mathfrak{a}},s\right)\tag{1.1},
\end{equation}
where $h=h_K$ is the class number, $w$ is the root of units of $K,$ $\mathfrak{a}$ is any primitive ideal, $z_{\mathfrak{a}}$ is the Heegner point, and $E(z,s)$ is the real analytic Eisenstein series of weight 0 for the modular group. The Eisenstein series $E(z,s)$ admits the Fourier expansion:
\begin{align*}
\Theta(s)E(z,s)=\Theta(s)y^{s}+\Theta(1-s)y^{1-s}+4y^{\frac12}\sum_{k=1}^{\infty}\sum_{mn=k}\left(\frac{m}n\right)^{it}K_{it}\left(2\pi ky\right)\cos\left(2\pi kx\right),
\end{align*}
where $\Theta(s):=\pi^{-s}\Gamma(s)\zeta(2s).$ Applying Fourier inversion we get from $(1.1)$ that
\begin{equation}
L_K\left(s,\chi\right)=w^{-1}\left(\frac{\sqrt{D}}2\right)^{-s}\zeta(2s)\sum_{z_{\mathfrak{a}\in\Lambda_D}}\bar{\chi}(\mathfrak{a})E\left(z_{\mathfrak{a}},s\right)\tag{1.2},
\end{equation}
where $\Lambda_D:=\left\{z_{\mathfrak{a}}\in \mathcal{F}:\ \text{$\mathfrak{a}$ primitive}\right\},$ and $\mathcal{F}$ is the fundamental domain for $SL_2(\mathbb{Z}).$

Clearly from the Fourier expansion we have $E(z,1/2)\equiv0,$ since $\zeta(2s)\sim\frac1{2s-1}$ when $s\rightarrow\frac12$ and the right hand side is well defined.

Thus take the derivative of $(1.2)$ at $s=\frac12$ and note $K_0(y)\ll y^{-\frac12}e^{-y}$ to get
\begin{align*}
L_K\left(\frac12,\chi\right)&=\frac{\sqrt{2}}w|D|^{-\frac14}\sum_{\mathfrak{a}}\bar{\chi}(\mathfrak{a})E'\left(z_{\mathfrak{a}},\frac12\right)\\
&=\frac1w\sum_{\mathfrak{a}}\frac{\bar{\chi}(\mathfrak{a})}{\sqrt{a}}\Big\{\log\frac{\sqrt{|D|}}{2a}+4\sum_{n=1}^{\infty}\tau(n)K_0\left(\frac{\pi n\sqrt{|D|}}{a}\right)\cos\left(\frac{\pi nb}{a}\right)\Big\}\\
&=\frac12\sum_{\mathfrak{a}}\frac{\bar{\chi}(a)}{\sqrt{a}}\log\frac{\sqrt{|D|}}{2a}+O\left(h(D)|D|^{-\frac14}\right),
 \end{align*}
since $\#\Lambda_D=h(D).$

Assuming $L(1/2,\chi_D)\geq0,$ i.e. $L_K(\frac12,\chi_0)=\zeta(\frac12)L(\frac12,\chi_D)\leq0,$
we derive
\begin{align*}
h(D)|D|^{-\frac14}&\gg\sum_{\mathfrak{a}}\frac1{\sqrt{a}}\log\frac{\sqrt{|D|}}{2a}
=\sum_{\mathfrak{a}}\frac1{\sqrt{a}}\log\frac{\sqrt{|D|}}{a}+O\left(h(D)|D|^{-\frac14}\right).
\end{align*}
Thus we have $h(D)|D|^{-\frac14}\gg\sum_{\mathfrak{a}}\frac1{\sqrt{a}}\log\frac{\sqrt{|D|}}{a}\gg\log |D|,$ which implies
\begin{thm}[Iwaniec-Kowalski]\cite{IK04}
Let notation be as before. Assume that $L(1/2,\chi_D)\geq0.$ Then we have
\begin{align*}
h(D)\gg|D|^{1/4}\log |D|.\tag{1.3}
\end{align*}
\end{thm}
\begin{remark}
The implied constant in $(1.3)$ is absolute. Actually, by estimating everything explicitly one can get an explicit lower bound:
\begin{thm}[\cite{DPS15}]
Let notation be as before. Assume $L(1/2,\chi_D)\geq0,$ then for any $\varepsilon\in(0,\frac12)$ we have
\begin{align*}
h(D)\geq 0.1265\varepsilon|D|^{1/4}\log |D|, \quad  \forall\; D\geq200^{\frac1{1-2\varepsilon}}.
\end{align*}
\end{thm}
\end{remark}

\subsection{Sketch of Proof of Theorem \ref{main} and Theorem \ref{c} in General CM Case}\label{sec2.2}
To generalize Iwaniec-Kowalski's results to CM fields case, one has to establish a general form of the Eisenstein period \eqref{1.1} (see below). The analogue of \eqref{1.1} in CM case is easy to build if $F,$ the totally real subfield, has trivial narrow class group (e.g. ref. \cite{Ma10}). In general, one needs to compute CM points associated to each Steinitz class. However, the situation is quite different from the imaginary quadratic case since generally, integral representatives in $Cl(F)$ bounded by Minkowski bounds may not necessarily be primitive. In addition, the CM type norm of the imaginary parts of CM points should be computed explicitly in order to compute the constant terms of Fourier coefficients of Eisenstein series. These problems are solved by the crucial Lemma \ref{prim} in Section \ref{sec3.2} below. Roughly speaking, fix a CM type $\Phi$ on $K,$ then a fractional ideal $\mathfrak{a}\subset K$ corresponds to a CM point $z_{\mathfrak{a}}$ in a Hilbert modular variety (ref. Section \ref{sec3}); we calculate the CM type norm $N_{\Phi}(y_{\mathfrak{a}})$ explicitly, where $y_{\mathfrak{a}}$ the imaginary part of $z_{\mathfrak{a}}.$
 Also, unlike the cases in \cite{IK04} or \cite{Ma10}, generally a single Hilbert Eisenstein series does not have a functional equation, but it turns out that it does vanish at the central point $s=1/2.$ Then we prove Proposition \ref{1}, which is an expression for $L_K(1/2,\chi)$ in terms of derivatives of Eisenstein series similar to \eqref{1.1}. Then following Iwaniec-Kowalski's idea, an explicit lower bound for $L_K(1/2,\chi_0)$ is given in Proposition \ref{fin}, Section \ref{sec4.1}. Furthermore, several effective estimates such as bounds for  $L_K(1,\chi)$ and the normalized Euler-Kronecker constant $\gamma_F^*$ are established in Section \ref{sec4}. With all these preparations, we eventually complete the proof of Theorem \ref{main} in Section \ref{sec4.2}. 
 
 The above mentioned results such as Proposition \ref{1} and Fourier expansion of derivatives of Hilbert Eisenstein series (i.e. Lemma \ref{Der}) will lead to a lower bound for the first moment of class group L-functions and an upper bound for the second moment for class group $L$-functions (ref. \eqref{xx} and \eqref{zz}). Then a standard technique using Cauchy inequality will imply $k=0$ case of Theorem \ref{c}. The $k=1$ case simply follows from the $k=0$ case and functional equation.

\section{Generalization to CM case}\label{sec3}
\subsection{Hilbert Modular Varieties and CM Zero-Cycles}\label{3}
\subsubsection{The Basic Correspondence}
Let $F/\mathbb{Q}$ be a totally real extension of degree $n.$ For any $S\subset F,$ let $S^{+}$ be the subset of $S$ consisting of totally positive elements. Given a fractional ideal $\mathfrak{f}\subset F,$ define
\begin{align*}
\Gamma(\mathfrak{f}):=\Bigg\{\gamma=\left(
                                      \begin{array}{cc}
                                        a & b \\
                                        c & d \\
                                      \end{array}
                                    \right)\in SL(2, F): \quad a,d\in \mathcal{O}_F,\; b\in\mathfrak{f},\; c\in\mathfrak{f}^{-1}
\Bigg\}.
\end{align*}

Let $\mathbb{H}$ be the upper half plane. Then $\Gamma(\mathfrak{f})$ acts on $\mathbb{H}^n$ via
$$
\gamma\cdot z=\left(\sigma_1(\gamma)z_1,\cdots,\sigma_n(\gamma)z_n\right),\quad\forall \; z=(z_1,\cdots,z_n)\in\mathbb{H}^n.
$$
Recall that the quotient
$$
X(\mathfrak{f}):=\Gamma(\mathfrak{f})\backslash\mathbb{H}^n
$$
is the open Hilbert modular variety associated to $\mathfrak{f}.$ It's known (\cite{Go06}, Theorem 2.17) that $X(\mathfrak{f})$ parameterizes isomorphism classes of triples $(A,i,m)$ where $(A,i)$ is an abelian variety with real multiplication $i:\; \mathcal{O}_F\hookrightarrow End(A)$ and
$$
m:\; \left(\mathfrak{M}_A,\mathfrak{M}^{+}_A\right)\rightarrow \left(\left(\mathfrak{o}_F\mathfrak{f}\right)^{-1},\left(\mathfrak{o}_F\mathfrak{f}\right)^{-1,+}\right)
$$
is an $\mathcal{O}_F$-isomorphism $\mathfrak{M}_A\xrightarrow{\sim}\left(\mathfrak{o}_F\mathfrak{f}\right)^{-1}$ which maps $\mathfrak{M}^{+}_A$ to $\left(\mathfrak{o}_F\mathfrak{f}\right)^{-1,+},$ where
$$
\mathfrak{M}_A:=\{\lambda:\; A\longrightarrow A^{\vee}\; \mid\; \lambda\; \text{is a symmetric $\mathcal{O}_F-$linear homomorphism}\}
$$
is the polarization module of $A$ and $\mathfrak{M}^{+}_A$ is its positive cone.

Let $K/F$ be a CM extension and let $\Phi=(\sigma_1,\cdots,\sigma_n)$ be a CM type of K. Then $z=(A,i,m)\in X(\mathfrak{f})$ is a CM point of type $(K,\Phi)$ is one of the following equivalent definitions holds:
\begin{enumerate}
  \item As a point $z\in \mathbb{H}^n,$ there is a point $\tau\in K$ such that
  $$
  \Phi(\tau)=(\sigma_1(\tau),\cdots,\sigma_n(\tau))=z
  $$
  and $\Lambda_{\tau}=\mathfrak{f}+\mathcal{O}_F\tau$ is a fractional ideal of $K.$
  \item $(A,i')$ is a CM abelian variety of type $(K,\Phi)$ with complex multiplication $i:\; \mathcal{O}_K\hookrightarrow End(A)$ such that $i=i'\mid_{\mathcal{O}_F}.$
\end{enumerate}
$\newline$
\noindent
To relate CM points with ideals of K, recall that we have fixed $\varepsilon_0\in K^{\times}$ such that $\overline{\varepsilon_0}=-\varepsilon_0$ and $\Phi(\varepsilon_0)\in \mathbb{H}^n.$ Let $\mathfrak{a}$ be a fractional ideal of $K$ and $\mathfrak{f}^{\mathfrak{a}}:=\varepsilon_0\mathfrak{o}_{K/F}\mathfrak{a}\overline{\mathfrak{a}}\cap F.$

Then by Lemma 3.1 in \cite{BY06}, the ideal class of $\mathfrak{f}^{\mathfrak{a}}$ is the Steinitz class of $\mathfrak{a}\subset K$ as a projective $\mathcal{O}_F-$module. Then it is clear that the CM abelian variety $\left(A_{\mathfrak{a}}=\mathbb{C}^n/\Phi(\mathfrak{a}),i\right)$ has the polarization module
$$
\left(\mathfrak{M}_A,\mathfrak{M}^{+}_A\right)\rightarrow \left(\left(\mathfrak{o}_F\mathfrak{f}^{\mathfrak{a}}\right)^{-1},\left(\mathfrak{o}_F\mathfrak{f}^{\mathfrak{a}}\right)^{-1,+}\right).
$$
To give an $\mathcal{O}_F-$isomorphism between the above pair and  $\left(\left(\mathfrak{o}_F\mathfrak{f}\right)^{-1},\left(\mathfrak{o}_F\mathfrak{f}\right)^{-1,+}\right)$ amounts to giving some $r\in F^{+}$ such that $\mathfrak{f}^{\mathfrak{a}}=r\mathfrak{f}.$ Therefore, to give a CM point $(A,i,m)\in X(\mathfrak{f})$ is the same as to give a pair $(\mathfrak{a},r),$ where $\mathfrak{a}$ is a fractional ideal of $K$ and $\mathfrak{f}^{\mathfrak{a}}=r\mathfrak{f}$ for some $r\in F^{+}.$ Two such pairs $(\mathfrak{a}_1,r_1)$ and $(\mathfrak{a}_2,r_2)$ are equivalent if there exists an $\gamma\in K^{\times}$ such that $\mathfrak{a}_2=\gamma\mathfrak{a}_1$ and $r_2=r_1\gamma\overline{\gamma}.$ We write $[\mathfrak{a},r]$ for the class of pair $(\mathfrak{a},r)$ and identify it with its associated CM point $(A_{\mathfrak{a}},i,m)\in X(\mathfrak{f}).$

Note that for any fractional ideal $\mathfrak{f}\subset K$ and any $r\in F^{+}$ we have the natural isomorphism of varieties:
\begin{align*}
\tau:\; X(\mathfrak{f})\xrightarrow{\qquad \sim \qquad}X(r \mathfrak{f}),\quad
z=(z_i)\longmapsto rz=\left(\sigma_i(r)z_i\right).
\end{align*}
\subsubsection{Upper Bounds of Minkowski Type and Steinitz Class}
Let $L/{\mathbb{Q}}$ be a number field of signature $(r_1, r_2).$ Let $n$ be the degree of $L/{\mathbb{Q}},$ then $n=r_1+2r_2.$ Minkowski showed that there is a constant $M(r_1, r_2)$ only depending on the signature such that for any $\mathcal{C}\in Cl(L),$ there exists an integral ideal $\mathfrak{a}_{\mathcal{C}}\in\mathcal{C}$ satisfying $N_{L/{\mathbb{Q}}}(\mathfrak{a}_{\mathcal{C}})\leq M(r_1, r_2)\sqrt{D_L},$ where $D_L$ is the absolute discriminant of $L.$

We can make the corollary in \cite{Zi81} (ref. P374) more explicit. In fact, an elementary estimate gives that
\begin{equation}\label{zi}
\log\frac{D_L}{N_{L/{\mathbb{Q}}}(\mathfrak{a}_{\mathcal{C}})}\geq(2\log2+\gamma)r_1+(\log2\pi+2\gamma)r_2-\sqrt{7n}.
\end{equation}
Then for $n\geq6,$ the right hand side is always positive since $\sqrt{n}\leq\sqrt{r_1}+\sqrt{2r_2}$. Let $M(n):=\frac{4^{r_2}n!}{\pi^{r_2}n^n}.$ Then $M(n)\cdot\sqrt{D_L}$ gives the Minkowski constant of $L/{\mathbb{Q}}.$ Combining $M(n)$ with \eqref{zi} one can take
\begin{equation}\label{zi*}
M(r_1, r_2):=\min\{e^{-(2\log2+\gamma)r_1-(\log2\pi+2\gamma)r_2+\sqrt{7n}}\cdot 1_{n\geq7}+ M(6)\cdot 1_{n\leq6},\; M(n)\}.
\end{equation}

In particular, when $n$ is large, we have $M(r_1, r_2)\leq50.7^{-r_1/2}\cdot19.9^{-r_2}.$ From now on, we shall fix $M(r_1,r_2)$ in \eqref{zi*}. For totally real extension $F/\mathbb{Q},$ each ideal class $\mathcal{C}\in Cl(F)$ contains an integral ideal $\mathfrak{f}_{\mathcal{C}}$ satisfying $N_{F/{\mathbb{Q}}}(\mathfrak{f}_{\mathcal{C}})\leq M(n,0)\sqrt{D_F}.$

Now we fix a set of fractional ideals
\begin{equation}\label{I_F}
\mathcal{I}^{+}_F:=\Big\{\mathfrak{f}:\; \mathfrak{f}\subset\mathcal{O}_F\; \text{and}\; N(\mathfrak{f})\leq M(n,0)\sqrt{D_F}\Big\},
\end{equation}
such that
$$
Cl(F)^{+}=\{[\mathfrak{f}]:\; \mathfrak{f}\in \mathcal{I}^{+}_F\},
$$
where $Cl(F)^{+}$ denotes the narrow ideal class group of $F.$ For simplicity, we assume $\mathcal{O}_F\in\mathcal{I}^{+}_F.$
Then for any fractional ideal $\mathfrak{a}\subset K,$ there exists a unique $\mathfrak{f}\in\mathcal{I}^{+}_F$ such that
$$
\mathfrak{f}^{\mathfrak{a}}:=\varepsilon_0\mathfrak{o}_{K/F}\mathfrak{a}\overline{\mathfrak{a}}\cap F\in[\mathfrak{f}],
$$
i.e., we can find some $r\in F^{+}$ such that $\mathfrak{f}^{\mathfrak{a}}=r\mathfrak{f}.$ Then by the above discussion, $[\mathfrak{a},r]$ gives a CM point in $X(\mathfrak{f}).$ Actually we can construct the CM point more explicitly. To achieve, let's recall the standard result:

\begin{prop}[\cite{Yo03}. Proposition 2.1. P179]\label{Yoo}
Let $F$ be an arbitrary algebraic number field and $K/F$ be an algebraic extension of degree $n.$ Let $\mathfrak{a}\subset K$ be a fractional ideal. Then there exist $\alpha_1,\cdots,\alpha_n\in F$ and a fractional ideal $\mathfrak{f}\subset F$ such that
\begin{align*}
\mathfrak{a}=\mathcal{O}_F\alpha_1\oplus\cdots\oplus\mathcal{O}_F\alpha_{n-1}\oplus\mathfrak{f}\alpha_n.
\end{align*}
Moreover, we have
$$
[\mathfrak{f}]=[\mathfrak{c}\cdot N_{K/F}(\mathfrak{a})]\in Cl(K),
$$
where $\mathfrak{c}$ is a fractional ideal of $F,$ independent of $\mathfrak{a},$ such that
$$
[\mathfrak{c}^2]=[D_{K/F}],\quad\text{where}\; D_{K/F}:=N_{K/F}(\mathfrak{o}_{K/F}).
$$
\end{prop}
\begin{proof}
The first part of the assertion comes from the structure theorem for a finitely generated torsion free module over a Dedekind domain. Then we canwrite
\begin{align*}
\mathcal{O}_K&=\mathcal{O}_F\alpha_1\oplus\cdots\oplus\mathcal{O}_F\alpha_{n-1}\oplus\mathfrak{c}\alpha_n;\\
\mathfrak{a}&=\mathcal{O}_F\beta_1\oplus\cdots\oplus\mathcal{O}_F\beta_{n-1}\oplus\mathfrak{f}\beta_n,
\end{align*}
where $\{\alpha_1,\cdots,\alpha_n\}$ and $\{\beta_1,\cdots,\beta_n\}$ are basis of $K$ over $F$ and $\mathfrak{f}$ and $\mathfrak{b}$ are fractional ideals of $F.$ Then there exists some $\gamma\in GL(n,F)$ such that $\gamma\alpha_i=\beta_i,$ $1\leq i\leq n.$ Take $x\in \mathbb{A}_F^{\times}$ such that $div(x)=\mathfrak{c}^{-1}\mathfrak{f}$ and set $y=diag[x,1,\cdots,1].$ Then clearly $y\gamma\mathcal{O}_K=\mathfrak{a}.$ On the other hand we have
$\mathfrak{a}=a\mathcal{O}_K,$ where $a\in\mathbb{A}_K^{\times}$ such that $div(a)=\mathfrak{a}.$ Therefore we have
$$
a^{-1}y\gamma\mathcal{O}_K=\mathcal{O}_K,
$$
which gives that $N_{K/F}(\mathfrak{a}^{-1}\mathfrak{c}^{-1}\mathfrak{f})=\det(\gamma)^{-1}N_{K/F}(a^{-1}y\gamma)\in\mathcal{O}_F^{\times},$ i.e.,
$$
[N_{K/F}(\mathfrak{a})]=[N_{K/F}(\mathfrak{c}^{-1}\mathfrak{f})]\in Cl(F).
$$
Therefore the last assertion is reduced to the case $\mathfrak{a}=\mathcal{O}_K.$

Let $\{\alpha'_1,\cdots,\alpha'_n\}$ be the dual basis of $K/F$ with respect to the relative trace $Tr_{K/F}.$ Then we have
$$
\mathfrak{o}_{K/F}^{-1}=\mathcal{O}_F\alpha'_1\oplus\cdots\oplus\mathcal{O}_F\alpha'_{n-1}\oplus\mathfrak{c}^{-1}\alpha'_n,
$$
where $\mathfrak{o}_{K/F}$ is the relative different with respect to $K/F.$
Then by the above discussion (i.e. taking $\mathfrak{a}=\mathfrak{o}_{K/F}^{-1}$) we have
$$
[N_{K/F}(\mathfrak{o}_{K/F}^{-1})]=[\mathfrak{c}^{-1}\cdot\mathfrak{c}^{-1}]\in Cl(K).
$$
Hence we have $[D_{K/F}]=[N_{K/F}(\mathfrak{o}_{K/F}^{-1})]=[\mathfrak{c}^{2}]\in Cl(K).$
\end{proof}

Let $\mathfrak{a}$ be any fractional ideal of $K,$ let $[\mathfrak{f}^{\mathfrak{a}}]$ be the Steinitz class of $\mathfrak{a}.$ Denote by 
\begin{equation}\label{10}
St_{\mathfrak{a}}=\Big\{\mathfrak{f}_{\mathfrak{a}}:\  N_{F/\mathbb{Q}}(\mathfrak{f}_{\mathfrak{a}})=\min_{\mathfrak{f}\in [\mathfrak{f}^{\mathfrak{a}}]}N_{F/\mathbb{Q}}(\mathfrak{f})\Big\}.
\end{equation}

Given a fractional ideal $\mathfrak{a}\subset K$, taking an $\mathfrak{f}_{\mathfrak{a}}\in St_{\mathfrak{a}}.$ Without loss of generality, we may assume that $\mathfrak{f}_{\mathfrak{a}}\in\mathcal{I}^{+}_F.$ We then fix this choice for any fractional ideal $\mathfrak{a}\in K$ once and for all. Then by Proposition \ref{Yoo} there is a decomposition
\begin{equation}\label{cm0}
\mathfrak{a}=\mathcal{O}_F\alpha\oplus\mathfrak{f}_{\mathfrak{a}}\beta.
\end{equation}

By the above proposition and the definition of $\mathfrak{f}^{\mathfrak{a}}$ we can take a appropriate $\beta$ such that there exists some $r\in F^{+}$ such that $\mathfrak{f}^{\mathfrak{a}}=r\mathfrak{f}_{\mathfrak{a}}.$

Define $z_{\mathfrak{a}}:=\frac{\alpha}{\beta}$. Then we have as in the proof of Lemma 3.2 of \cite{BY06} that
$$
(\overline{\alpha}\beta-\alpha\overline{\beta})\mathfrak{f}\mathcal{O}_K=\mathfrak{o}_{K/F}\mathfrak{a}\overline{\mathfrak{a}}.
$$
Then we have
$$
\varepsilon_0(\overline{\alpha}\beta-\alpha\overline{\beta})=r\varepsilon \quad\text{for some}\; \varepsilon\in\mathcal{O}_F^{\times.}
$$
Replacing $\beta$ by $\varepsilon^{-1}\beta$ if necessary, we can assume $\varepsilon=1.$ This implies that
$$
\varepsilon_0(\bar{z}-z)=\frac{r}{\beta\overline{\beta}}\in F^{\times},
$$
and thus $z_{\mathfrak{a}}\in K^{\times}\cap\mathbb{H}^n=\{z\in K^{\times}: \Phi(z)\in \mathbb{H}^n\}$.
Moreover, $z$ represents the CM point $[\mathfrak{a},r]\in X(\mathfrak{f}_{\mathfrak{a}}).$

Let $\mathcal{CM}(K,\Phi,\mathfrak{f})$ be the set of CM points $[\mathfrak{a},r]\in X(\mathfrak{f})$ which we regard as a CM 0-cycle in $X(\mathfrak{f}).$ Let
$$
\mathcal{CM}(K,\Phi):=\sum_{[\mathfrak{f}]\in Cl(F)^{+}}\mathcal{CM}(K,\Phi,\mathfrak{f}).
$$
We have the natural surjective map
\begin{align*}
\mathcal{CM}(K,\Phi)\twoheadrightarrow Cl(K),\quad
[\mathfrak{a},r]\mapsto [\mathfrak{a}].
\end{align*}
The fiber is indexed by $\epsilon\in \mathcal{O}^{\times,+}_F/ N_{K/F}\mathcal{O}^{\times}_K,$ since every element in the fiber of $\mathfrak{a}$ is of the form $[\mathfrak{a},r\epsilon]$ with $r$ fixed and $\epsilon\in \mathcal{O}^{\times,+}_F$ a totally positive unit. Note that $\sharp\left(\mathcal{O}^{\times,+}_F/ N_{K/F}\mathcal{O}^{\times}_K\right)\leq 2.$

\subsection{Representation of Ideals}\label{sec3.2}
Let $z_{\mathfrak{a}}$ be the CM point corresponding to the fractional ideal $\mathfrak{a}.$ Write $x_{\mathfrak{a}}$ (resp. $y_{\mathfrak{a}}$) to be the real part (resp. imaginary part) of $z_{\mathfrak{a}}.$ To prove our main results, we need to compute $y_{\mathfrak{a}}$ explicitly. We start with recalling some definition.
\begin{defn}[Primitive Ideals]
Let $\mathfrak{a}$ be a fractional ideal of $\mathcal{O}_K.$ We say that $\mathfrak{a}$ is primitive if $\mathfrak{a}$ is an integral ideal of $\mathcal{O}_K$ and if for any nontrivial integral ideal $\mathfrak{n}$ of $\mathcal{O}_F,$ $\mathfrak{n}^{-1}\mathfrak{a}$ is not an integral ideal.
\end{defn}
\begin{fact}
For any fractional ideal $\mathfrak{a}$ of $\mathcal{O}_K,$ there exists a unique fractional ideal $\mathfrak{n}$ of $F$ such that $\mathfrak{n}^{-1}\mathfrak{a}$ is a primitive ideal. The ideal $\mathfrak{n}$ will be called the content of the ideal $\mathfrak{a}.$
\end{fact}

Let $K/F$ is a CM extension. There exists some $D\in F^{\times}/\left(F^2\cap F^{\times}\right)$ such that $K=F(\sqrt{D}).$ We may assume $D\in\mathcal{O}_F$ and fix this choice once for all. Let $\mathfrak{q}$ be the index-ideal $\Big[\mathcal{O}_K: \mathcal{O}_F[\sqrt{D}]\Big].$ Set $\widetilde{\mathfrak{q}}=\mathfrak{q}\mathcal{O}_K.$ 

\begin{prop}[Section 2.6 of \cite{Co00}]\label{CM}
Let $\mathfrak{a}$ be a fractional ideal of $K.$ There exist unique ideals $\mathfrak{n}$ and $\mathfrak{m}$ and an element $b\in\mathcal{O}_F$ such that
\begin{equation}\label{cm}
\mathfrak{a}=\mathfrak{n}\left(\mathfrak{m}\oplus\mathfrak{q}^{-1}(-b+\sqrt{D})\right),
\end{equation}
where $\mathfrak{q}$ is the index-ideal $\Big[\mathcal{O}_K: \mathcal{O}_F[\sqrt{D}]\Big].$
In addition, we have the following:
\begin{description}
  \item[1] $\mathfrak{n}$ is the content of $\mathfrak{a}.$
  \item[2] $\mathfrak{a}$ is an integral ideal of $\mathcal{O}_K$ if and only if $\mathfrak{n}$ is an integral ideal of $\mathcal{O}_F.$
  \item[3] $\mathfrak{a}$ is primitive in $K/F$ if and only if $\mathfrak{n}=\mathcal{O}_F.$
  \item[4] $\mathfrak{m}$ is an integral ideal and $\mathfrak{a}\overline{\mathfrak{a}}=\mathfrak{m}\mathfrak{n}^2.$
\end{description}
\end{prop}

\begin{remark}
$b$ is determined by the modulo relation
\begin{align*}
\left\{
  \begin{array}{ll}
    \delta-b\in\mathfrak{q}, \\
    b^2+D\in \mathfrak{m}\mathfrak{q}^2,
  \end{array}
\right.
\end{align*}
where $\delta\in\mathcal{O}_F$ comes from the corresponding pseudo-matrix on the basis $(1,\sqrt{D})$ (ref. \cite{Co00}, Corollary 2.2.9).
\end{remark}

The equations \eqref{cm0} and \eqref{cm} give us two decompositions of a fractional ideal $\mathfrak{a}$ of $K.$ However, the main obstacle comes from the factor $\mathfrak{n}$ in \eqref{cm}. We may not easily get rid of $\mathfrak{n}$ unless the ideal class group $Cl(F)$ is trivial. Noting that $\mathfrak{n}$ is a content, one natural way is to use the decompositions to construct a group of primitive representatives of the ideal class group $Cl(K)$ such that the CM norms of the imaginary part of the corresponding CM points can be computed explicitly. In fact, It can be seen from the definition that an integral ideal $\mathfrak{a}$ of $\mathcal{O}_K$ is primitive if and only if its primary decomposition is of the following form:
$$
\mathfrak{a}=\prod_{j}\mathfrak{P}'_j\cdot\prod_{i}\mathfrak{P}_i^{\alpha_i}\overline{\mathfrak{P}}_i^{\beta_i},
$$
where $\mathfrak{P}'_j$ are ramified primes and $\mathfrak{P}_i$ are splitting primes with $\alpha_i\cdot\beta_i=0.$ In particular, every split prime ideal of $\mathcal{O}_K$ is primitive. On the other hand, by $\check{C}eboyarev$ density theorem, there exist a group of representatives of $Cl(K)$ consisting of split prime ideals. This gives us a set of primitive representatives of $Cl(K).$ However, since we have to bound these representatives uniformly (as can be seen in the last section) and it is not easy to give such a bound for splitting ideals in each ideal class, we move on in another way.

It's well known that, for any fractional $\mathcal{O}_F-$ideals $\mathfrak{a}$ and $\mathfrak{b},$ we have the isomorphism $\mathfrak{a}\oplus\mathfrak{b}\simeq \mathcal{O}_F\oplus\mathfrak{a}\mathfrak{b}.$ But this is not enough, to make \eqref{cm} into the form of \eqref{cm0}, we need to make the isomorphism into an identity.
\begin{lemma}\label{decom}
Suppose $K/F$ is a finite extension of number fields. Let $\mathfrak{a}$ and $\mathfrak{b}$ be fractional ideals of $\mathcal{O}_F.$ Let $\alpha,$ $\beta$ be two elements in $K^{\times}.$ Assume that $a\in \mathfrak{a},$ $b\in \mathfrak{b},$ $c\in \mathfrak{b}^{-1}$ and $d\in \mathfrak{a}^{-1}$ such that $ad-bc=1\in F.$ Set
$$
(\alpha',\beta'):=(\alpha,\beta)\left(
                                  \begin{array}{cc}
                                    a & c \\
                                    b & d \\
                                  \end{array}
                                \right),
$$
then we have
$$
\mathfrak{a}\alpha+\mathfrak{b}\beta=\mathcal{O}_F\alpha'+\mathfrak{a}\mathfrak{b}\beta'.
$$
\end{lemma}
\begin{proof}
We have $\alpha'=a\alpha+b\beta$ and $\beta'=c\alpha+d\beta.$ Hence
$$
\mathcal{O}_F\alpha'+\mathfrak{a}\mathfrak{b}\beta'\subset(\mathcal{O}_F\cdot a+\mathfrak{a}\mathfrak{b}\cdot c)\alpha+(\mathcal{O}_F\cdot b\mathfrak{a}\mathfrak{b}\cdot d)\beta\subset\mathfrak{a}\alpha+\mathfrak{b}\beta.
$$
Conversely, we have $\alpha=d\alpha'-b\beta'$ and $\beta=-c\alpha'+a\beta'.$ Hence
$$
\mathfrak{a}\alpha+\mathfrak{b}\beta\subset\mathcal{O}_F\alpha'+\mathfrak{a}\mathfrak{b}\beta'.
$$
\end{proof}

Let $z_{\mathfrak{a}}$ be the associate CM point to $\mathfrak{a}$. Define the CM type norm of $y_{\mathfrak{a}}$ as $N_{\Phi}(y_{\mathfrak{a}}):=\prod_{\sigma\in\Phi}\sigma(y_{\mathfrak{a}}).$ In Section \ref{3.3} we will see that CM type norms show up naturally in Fourier coefficients of Hilbert Eisenstein series and their derivatives at the central value. According the period formula \eqref{l} CM type norms of imaginary parts of CM points also connect with central values of class group $L$-functions. By the above preparation we can prove an explicit expression of $N_{\Phi}(y_{\mathfrak{a}})$ as follows:

\begin{prop}\label{prim}
	Let notation be as before. Then we have
	\begin{equation}\label{cm2}
	N_{\Phi}(y_{\mathfrak{a}})=\frac{N_{K/{\mathbb{Q}}}(c_{\mathfrak{a}})N_{F/{\mathbb{Q}}}(\mathfrak{f}_{\mathfrak{a}})N_{F/{\mathbb{Q}}}(\mathfrak{q})^2}
	{2^nN_{K/{\mathbb{Q}}}(\mathfrak{a})}\cdot\frac{\sqrt{D_K}}{D_F},
	\end{equation}
	where $\mathfrak{f}_{\mathfrak{a}}\in St_{\mathfrak{a}},$ $c_{\mathfrak{a}}$ is an element in the content of $\mathfrak{a}\widetilde{\mathfrak{q}}^{-1}$ that is of the minimal absolute norm, and $D_K$ (resp. $D_F$) is the absolute discriminant of $K/{\mathbb{Q}}$ (resp. $F/{\mathbb{Q}}$).
\end{prop}
\begin{proof}
By the argument in the above remark, we may assume $\mathfrak{a}$ is an integral ideal of $\mathcal{O}_K$ such that $\mathfrak{a}\widetilde{\mathfrak{q}}^{-1}$ is integral. Let $\mathfrak{n}$ be the content of $\mathfrak{a}\widetilde{\mathfrak{q}}^{-1},$ then $\mathfrak{n}$ is integral. Noting that $\mathfrak{q}\subset\mathcal{O}_F,$ hence by \eqref{cm} we have the decomposition
\begin{equation}\label{cm1}
\mathfrak{a}=\mathfrak{n}\cdot(-b+\sqrt{D})\oplus \mathfrak{n}^{-1}\overline{\mathfrak{q}}^{-1}\mathfrak{a}\overline{\mathfrak{a}}=\mathfrak{n}\cdot(-b+\sqrt{D})\oplus \mathfrak{n}^{-1}\mathfrak{q}^{-1}\mathfrak{a}\overline{\mathfrak{a}},
\end{equation}
where $b\in\mathcal{O}_F$ and $\mathfrak{q}$ is the index-ideal $\Big[\mathcal{O}_K: \mathcal{O}_F[\sqrt{D}]\Big].$

Let $c_{\mathfrak{a}}\in \mathfrak{n}$ be an element of the minimal absolute norm, and we fix one such choice for each $\mathfrak{a}$ once and for all. Then by Lemma \ref{decom} we have
\begin{equation}\label{23}
\mathfrak{a}=\mathfrak{n}\cdot(-b+\sqrt{D})\oplus \mathfrak{n}^{-1}\mathfrak{q}^{-1}\mathfrak{a}\overline{\mathfrak{a}}
=\mathcal{O}_F\cdot c_{\mathfrak{a}}\cdot (-b+\sqrt{D})\oplus\mathfrak{q}^{-1}\mathfrak{a}\overline{\mathfrak{a}}\cdot c_{\mathfrak{a}}^{-1}.
\end{equation}
The direct sum in the right hand side of the above identity can be verified easily from the proof of Lemma \ref{decom}. Also noting that by the definition of $\mathfrak{q}$ we have $\mathfrak{o}_{K/F}=4D\mathfrak{q}^{-2},$ where $\mathfrak{o}_{K/F}$ is the relative ideal-discriminant, then $[\mathfrak{q}^{-1}\mathfrak{a}\overline{\mathfrak{a}}]$ is the Steinitz class of $\mathfrak{a}.$

Combining the decomposition \eqref{cm1} with \eqref{cm0}, i.e. $\mathfrak{a}=\mathcal{O}_F\alpha\oplus\mathfrak{f}_{\mathfrak{a}}\beta,$ we have, by the uniqueness of Steinitz class, that
$$
\alpha=c_{\mathfrak{a}}\cdot(-b+\sqrt{D})\varepsilon\quad \text{and}\; \mathfrak{f}_{\mathfrak{a}}\beta=\mathfrak{q}^{-1}\mathfrak{a}\overline{\mathfrak{a}}\cdot c_{\mathfrak{a}}^{-1},
$$
for some unit $\varepsilon\in\mathcal{O}_F^{\times}.$ So we have
$$
y_{\mathfrak{a}}=\mathfrak{Im}(z_{\mathfrak{a}})=\frac{c_{\mathfrak{a}}\cdot \mathfrak{Im}(\sqrt{D})}{\beta}.
$$
Noting that $\mathfrak{o}_{K/F}=N_{K/F}(\mathfrak{o}_{K/F})$ and $D_K=D_F^2N_{F/{\mathbb{Q}}}(\mathfrak{o}_{K/F}),$
one thus obtains
\begin{align*}
N_{\Phi}(y_{\mathfrak{a}})=\prod_{\sigma\in\Phi}\left(\frac{c_{\mathfrak{a}}\cdot \sqrt{D}}{\beta}\right)
=\frac{N_{K/{\mathbb{Q}}}(c_{\mathfrak{a}})N_{F/{\mathbb{Q}}}(\mathfrak{f}_{\mathfrak{a}})N_{F/{\mathbb{Q}}}(\mathfrak{q})^2}
{2^nN_{K/{\mathbb{Q}}}(\mathfrak{a})}\cdot\frac{\sqrt{D_K}}{D_F}.
\end{align*}
\end{proof}

\begin{remark}
From the above expression, it is clear that $N_{\Phi}(y_{\mathfrak{a}})$ is independent of a particular choice of $\mathfrak{f}_{\mathfrak{a}}\in St_{\mathfrak{a}}.$ Also, the term $N_{K/{\mathbb{Q}}}(c_{\mathfrak{a}})$ in the right hand side of \eqref{cm2} does not depend on a particular choice of $c_{\mathfrak{a}}.$ In fact \eqref{cm2} shows that $N_{\Phi}(y_{\mathfrak{a}})$ is independent of the choice of a particular representative of the class $[\mathfrak{a}].$ This is because the factors $N_{F/{\mathbb{Q}}}(\mathfrak{f}_{\mathfrak{a}})$ and $N_{K/{\mathbb{Q}}}(c_{\mathfrak{a}}\mathfrak{a}^{-1})$ are both invariant under scalar multiplication by $K^{\times}.$
\end{remark}

We will always fix the CM type $\Phi$ in this paper. For the sake of simplicity, we will write $y_{\mathfrak{a}}^{\bm{\sigma}}$ for the CM type norm $N_{\Phi}(y_{\mathfrak{a}})$ in computations in the following parts. 

\subsection{Hilbert Eisenstein Series}\label{3.3}
Let notation be as before. Let $\mathfrak{a}$ and $\mathfrak{b}$ be fractional ideals of $F.$ Take $\varphi$ to be the characteristic function of the closure of $\mathfrak{a}\mathfrak{b}\oplus\mathfrak{b}.$ Let $\varphi_{\mathfrak{a}\mathfrak{b}}$ be the characteristic function of the closure of $\mathfrak{a}\mathfrak{b},$ and $\varphi_{\mathfrak{b}}$ be the characteristic function of the closure of $\mathfrak{b}.$ Then we define
\begin{align*}
G_{\mathbf{k}}(z,s;\varphi)=y^{-\frac{\mathbf{k}}2+s\bm{\sigma}}\sum_{(c,d)\in F^{2,\times}/\mathcal{O}^{\times}_F}\varphi(c,d)(cz+d)^{-\mathbf{k}}|cz+d|^{\mathbf{k}-2s\bm{\sigma}}.
\end{align*}
Define
$$
\Gamma_{\mathfrak{a}}:=\Bigg\{\gamma\in\left(
                                         \begin{array}{cc}
                                           a & b \\
                                           c & d \\
                                         \end{array}
                                       \right)\in GL(2,F):\; a,d\in\mathcal{O}_F,\; b\in\mathfrak{a}^{-1},\; c\in\mathfrak{a},\; det \gamma\in\mathcal{O}_F^{+}
\Bigg\}.
$$
Then clearly $\varphi(x\gamma)=\varphi(x)$ for $x\in F\oplus F,$ $\gamma\in\Gamma_{\mathfrak{a}}.$ Form now on, we assume $\mathbf{k}=0,$ then one can check that we have
$$
G_{\mathbf{k}}(\gamma z, s;\mathfrak{a},\mathfrak{b})=G_{\mathbf{k}}(z, s;\mathfrak{a},\mathfrak{b}),\quad\forall\; \gamma\in\Gamma_{\mathfrak{a}}.
$$

Let $G(z,s;\mathfrak{a},\mathfrak{b}):=G_0(z,s;\mathfrak{a},\mathfrak{b}),$ define the regularized Eisenstein series as
$$
E(z,s;\mathfrak{a},\mathfrak{b}):=\zeta_F(2s)^{-1}G(z,s;\mathfrak{a},\mathfrak{b}),\ \Re(s)>1.
$$
Then based on the Fourier expansion of $G(z,s;\mathfrak{a},\mathfrak{b})$ (ref. Chapter V of \cite{Yo03}) we have the explicit Fourier expansion:
\begin{align*}
E(z,s;\mathfrak{a},\mathfrak{b})&=N(\mathfrak{b})^{-2s}y^{s\bm{\sigma}}\frac{\zeta_F(2s,[\mathfrak{b}]^{-1})}{\zeta_F(2s)}\\
&\quad+\left(\frac{\sqrt{\pi}\Gamma(s-1/2)}{\Gamma(s)}\right)^nD^{-\frac12}_FN(\mathfrak{b})^{-1}y^{(1-s)\bm{\sigma}}N(\mathfrak{a}\mathfrak{b})^{1-2s}
\frac{\zeta_F(2s-1,[\mathfrak{a}\mathfrak{b}]^{-1})}{\zeta_F(2s)}\\
&\quad+\left(\frac{2\pi^s}{\Gamma(s)}\right)^n\frac{y^{\frac{\bm{\sigma}}2}}{D^{1/2}_FN(\mathfrak{b})\zeta_F(2s)}\sum_{b\in F^{\times}}|N(b)|^{s-1/2}\lambda(b,s)\mathbf{e}(bx)\\
&\qquad\times \prod_{v\in\mathbf{J}_{\infty}}K_{s-\frac12}\left(2\pi y_v|b_v|\right).
\end{align*}

Since we have the following Laurant expansion of (partial) Dedekind zeta function around $s=1:$
\begin{equation}\label{zeta}
\zeta_F(s,\mathcal{C})=\frac{h^{-1}_F\rho_F}{s-1}+\gamma_{F,[\mathcal{C}]}+O(s-1),
\end{equation}
where $\mathcal{C}$ is an ideal class in $Cl(F),$ and
$\rho_F=2^nh_FR_Fw_F^{-1}{D_F}^{-1/2}$ is the residue of $\zeta_{F}(s)$ at $s=1.$ In particular, around $s=1$ we have
\begin{equation}\label{zeta1}
\zeta_F(s)=\frac{\rho_F}{s-1}+\gamma_{F}+O(s-1),
\end{equation}
where
$$
\gamma_{F,\mathcal{C}}:=\lim_{s\mapsto1}\Big\{\zeta_F(s,\mathcal{C})-\frac{h^{-1}_F\rho_F}{s-1}\Big\};
$$
and by $\zeta_F(s)=\sum_{[\mathcal{C}]\in Cl(F)}\zeta_F(s,\mathcal{C})$ we have
$$
\gamma_{F}=\sum_{\mathcal{C}\in Cl(F)}\gamma_{F,\mathcal{C}}=\rho_F\gamma_F^*,
$$
where $\gamma_F^*$ is defined in \eqref{normized}. $\gamma_F$ and $\gamma_{F,\mathcal{C}}$ are called unnormalized Euler-Kronecker constants with respect to $F/\mathbb{Q},$ which we will deal with later.

From the Fourier expansion above we see that $E(z,s;\mathfrak{a},\mathfrak{b})$ has a meromorphic continuation to $\mathbb{C}$ with a simple pole at $s=1$ with residue
$$
Res_{s=1}E(z,s;\mathfrak{a},\mathfrak{b})=\frac{2^{n-1}\pi^nR_F}{w_FD_FN(\mathfrak{b})N(\mathfrak{a}\mathfrak{b})\zeta_F(2)}.
$$
Note that we have the Taylor expansion around $s=0$ that $\Gamma(s)=s^{-1}+O(1)$ and
\begin{equation}\label{zeta2}
\zeta_F(s)=-\frac{h_FR_F}{w_F}s^{n-1}+O(s^n),
\end{equation}
where $R_F$ is the regulator and $w_F$ is the number of roots of unity. Then $E(z,s;\mathfrak{a},\mathfrak{b})$ is holomorphic at $s=1/2.$
Moreover, we can show actually $E(z,s;\mathfrak{a},\mathfrak{b})$ vanishes at $s=1/2,$ $\forall$ $z\in \mathbb{C}.$
\begin{lemma}\label{0}
Let notation be as above. Then we have
$$
E\left(z,\frac12;\mathfrak{a},\mathfrak{b}\right)\equiv0,\quad \forall\; z\in\mathbb{H}^n.
$$
\end{lemma}
\begin{proof}
Let $\mathfrak{f}\subset F$ be a fractional ideal, and $\tilde{\mathfrak{f}}$ be its dual, i.e. $[\mathfrak{f}]\cdot[\tilde{\mathfrak{f}}]=[\mathfrak{D}],$ where $\mathfrak{D}$ is the different of $F/\mathbb{Q}.$ Let 
$$\mathcal{Z}_F(s,[\mathfrak{f}]):=\mathcal{Z}_{F,\infty}(s)\zeta_F(s,[\mathfrak{f}]),\; \text{and}\; \mathcal{Z}_F(s):=\mathcal{Z}_{F,\infty}(s)\zeta_F(s),$$
where $\mathcal{Z}_{F,\infty}(s):={D_F}^{s/2}{\pi}^{-ns/2}\Gamma\left(s/2\right)^n.$ It is well known that we have the following functional equation for the partial completed zeta function:
\begin{equation}\label{2.2.4}
\mathcal{Z}_F(s,[\mathfrak{f}])=\mathcal{Z}_F(1-s,[\tilde{\mathfrak{f}}]).
\end{equation}
An immediate consequence of this is the functional equation for the completed Dedekind zeta function (obtained adding the partial ones), which has exactly the same form. Also, all the partial zeta functions have a simple pole at $s=1$ with the same residue $2^nR_Fw_F^{-1}D_F^{-1/2}.$

Let's introduce some functions to simplify the notations. Define
\begin{align*}
M_1(s):&=\mathcal{Z}_F(2s)N(\mathfrak{b})^{-2s}y^{s\bm{\sigma}}\frac{\zeta_F(2s,[\mathfrak{b}]^{-1})}{\zeta_F(2s)}
=\mathcal{Z}_{F,\infty}(2s)N(\mathfrak{b})^{-2s}y^{s\bm{\sigma}}\zeta_F(2s,[\mathfrak{b}]^{-1});\\
M^{*}_2(s):&=\left(\frac{\sqrt{\pi}\Gamma(s-1/2)}{\Gamma(s)}\right)^nD^{-\frac12}_FN(\mathfrak{b})^{-1}y^{(1-s)\bm{\sigma}}N(\mathfrak{a}\mathfrak{b})^{1-2s}
\frac{\zeta_F(2s-1,[\mathfrak{a}\mathfrak{b}]^{-1})}{\zeta_F(2s)}\\
M_2(s):&=\mathcal{Z}_F(2s)M^{*}_2(s)\\
&=\mathcal{Z}_{F,\infty}(2s-1)N(\mathfrak{b})^{-1}y^{(1-s)\bm{\sigma}}N(\mathfrak{a}\mathfrak{b})^{1-2s}\zeta_F(2s-1,[\mathfrak{a}\mathfrak{b}]^{-1}).\\
\end{align*}
Then by \eqref{2.2.4} (taking $\mathfrak{f}=\mathfrak{a}\mathfrak{b}$) we see that $M_2(1-s)=H(s)M_1(s),$ where
\begin{align*}
H(s):=N(\mathfrak{b})^{2s-1}N(\mathfrak{a}\mathfrak{b})^{2s-1}\cdot\frac{\zeta_F(2s,[\widetilde{\mathfrak{a}\mathfrak{b}}]^{-1})}
{\zeta_F(2s,[\mathfrak{b}]^{-1})}.
\end{align*}
Since the Dirichlet series $\zeta_F(2s,[\mathfrak{b}]^{-1})$ absolutely converges when $\mathfrak{Re}(s)>\frac12,$ $H(s)$ is holomorphic when $\mathfrak{Re}(s)>\frac12,$ and can be continued to a meromorphic function on $\mathbb{C}$. Since all the partial zeta functions have a simple pole at $s=1$ with the same residue, we see $H(s)$ is holomorphic at $s=\frac12$ and $H(\frac12)=1.$

Let $L(s)^{-1}:=H(s)H(1-s),$ then we have
\begin{align*}
L(s)^{-1}=\frac{\zeta_F(2s,[\widetilde{\mathfrak{a}\mathfrak{b}}]^{-1})\zeta_F(2-2s,[\widetilde{\mathfrak{a}\mathfrak{b}}]^{-1})}
{\zeta_F(2s,[\mathfrak{b}]^{-1})\zeta_F(2-2s,[\mathfrak{b}]^{-1})}.
\end{align*}
Likewise, $L(s)$ is a meromorphic function on $\mathbb{C}$ and is analytic at $s=1/2,$ with $L(1/2)=1.$ So we have $M(1-s)=H(s)M_1(S)+H(s)L(s)M_2(s).$ Now let
\begin{align*}
E_1(s):=\mathcal{Z}_F(2s)\left(\frac{2\pi^s}{\Gamma(s)}\right)^n\frac{y^{\frac{\bm{\sigma}}2}}{D^{1/2}_FN(\mathfrak{b})\zeta_F(2s)}
=2^nD^s_F\frac{y^{\frac{\bm{\sigma}}2}}{D^{1/2}_FN(\mathfrak{b})};
\end{align*}
\begin{align*}
E_2(s):&=\sum_{b\in F^{\times}}|N(b)|^{s-1/2}\lambda(b,s)\mathbf{e}(bx)\prod_{v\in\mathbf{J}_{\infty}}K_{s-\frac12}\left(2\pi y_v|b_v|\right)\\
&=\sum_{b\in F^{\times}}|N(b)|^{s-1/2}\mathop{\sum\sum}_{\substack{(a,c)\in\mathfrak{b}^{-1}\mathfrak{o}^{-1}\times\mathcal{C}\\ac=b}}|N(c)|^{1-2s}
\mathbf{e}(bx)\prod_{v\in\mathbf{J}_{\infty}}K_{s-\frac12}\left(2\pi y_v|b_v|\right)\\
&=\sum_{b\in F^{\times}}\mathop{\sum\sum}_{\substack{(a,c)\in\mathfrak{b}^{-1}\mathfrak{o}^{-1}\times\mathcal{C}\\ac=b}}\left(\frac{|N(a)|}{|N(c)|}\right)^{s-\frac12}
\mathbf{e}(bx)\prod_{v\in\mathbf{J}_{\infty}}K_{s-\frac12}\left(2\pi y_v|b_v|\right);\\
\end{align*}
Also set $E(s)=E_1(s)E_2(s).$ In fact both $E_1(s)$ and $E_2(s)$ are entire functions. Let's briefly explain why $E_2(s)$ is entire: by Turan's inequality for Bessel functions, $\log K_{\nu}(x)$ is convex. Also note the fact that
$$
\lim_{x\mapsto+\infty}\frac{\log K_{\nu}(x)}{x}=-1,
$$
hence we have $K_{\nu}(x)\leq(eK_{\nu}(1))e^{-x}.$ Namely, the Bessel K-function has exponentially decay, which forces the sum in $E_2(s)$ to converge absolutely. Also one sees easily that $E_1(s)=D^{2s-1}_FE_1(1-s).$ So we have
$$
E(1-s)=D^{1-2s}_FE_1(s)E_2(1-s).
$$
Let $\mathcal{E}(z,s;\mathfrak{a},\mathfrak{b}):=\mathcal{Z}_F(2s)E(z,s;\mathfrak{a},\mathfrak{b})$ be completed Eisenstein series, then by the Fourier expansion of $E(z,s;\mathfrak{a},\mathfrak{b})$ we have
$$
\mathcal{E}(z,s;\mathfrak{a},\mathfrak{b})=M(s)+E(s)=M_1(s)+M_2(s)+E_1(s)E_2(s).
$$
By the above computation one has
$$
\mathcal{E}(z,1-s;\mathfrak{a},\mathfrak{b})=H(s)M_1(s)+L(s)H(s)M_2(s)+D^{1-2s}_FE_1(s)E_2(1-s).
$$
Therefore, we have (noting that $K_{\nu}(x)=K_{-\nu}(x)$)
\begin{align*}
E(z,\frac12;\mathfrak{a},\mathfrak{b})&=\lim_{s\mapsto\frac12}\frac{\mathcal{E}(z,s;\mathfrak{a},\mathfrak{b})}{\mathcal{Z}_{F,\infty}(2s)\zeta_F(2s)}\\
&=\lim_{s\mapsto\frac12}\frac{\mathcal{E}(z,1-s;\mathfrak{a},\mathfrak{b})}{\mathcal{Z}_{F,\infty}(2-2s)\zeta_F(2-2s)}\\
&=-\lim_{s\mapsto\frac12}\frac{H(s)M_1(s)+L(s)H(s)M_2(s)+D^{1-2s}_FE_1(s)E_2(1-s)}{\mathcal{Z}_{F,\infty}(2-2s)\zeta_F(2s)}\\
&=-\lim_{s\mapsto\frac12}\frac{H(\frac12)M_1(s)+L(\frac12)H(\frac12)M_2(s)+E_1(\frac12)E_2(\frac12)}{\mathcal{Z}_{F,\infty}(2s)\zeta_F(2s)}\\
&=-\lim_{s\mapsto\frac12}\frac{\mathcal{E}(z,s;\mathfrak{a},\mathfrak{b})}{\mathcal{Z}_{F,\infty}(2s)\zeta_F(2s)}=-E(z,\frac12;\mathfrak{a},\mathfrak{b}).
\end{align*}
Thus we have $E(z,\frac12;\mathfrak{a},\mathfrak{b})=0.$
\end{proof}

\begin{remark}
Note that $E(z,s;\mathfrak{a},\mathfrak{b})$ may not have a functional equation, since the Hilbert modular variety may have several cusps. The Eisenstein matrix will always have a functional equation. However,
when $\mathfrak{a}=\mathfrak{b}=\mathcal{O}_F,$ there is only one cusp. In fact, we see that in this situation $H(s)=L(s)=1$ and $E(s)=E(1-s),$ then we have the functional equation
$$
\mathcal{E}(z,s;\mathfrak{a},\mathfrak{b})=\mathcal{E}(z,1-s;\mathfrak{a},\mathfrak{b}),
$$
which gives directly that $E(z,\frac12;\mathfrak{a},\mathfrak{b})=0.$ This is the case in \cite{IK04}.
\end{remark}

\subsection{Periods of Eisenstein Series}
In this section we combine the discussion in last two subsections to show the class group $L$-function $L_K(\chi,s)$ can be expressed as a weighted period of the Eisenstein series $E(z,s;\mathfrak{f})$ with respect to the CM 0-cycles $\mathcal{CM}(K,\Phi,\mathfrak{f}),$ where $[\mathfrak{f}]\in Cl(F)^{+}.$

Recall that we have the natural surjective map
\begin{align*}
\mathcal{CM}(K,\Phi)\twoheadrightarrow Cl(K),\quad [\mathfrak{a},r]\mapsto [\mathfrak{a}].
\end{align*}
And the fiber is indexed by $\epsilon\in \mathcal{O}^{\times,+}_F/ N_{K/F}\mathcal{O}^{\times}_K$ with order at most 2.

Recall that by Lemma \ref{prim},
Since $K/F$ is a CM extension of number fields of degree $2n,$ we have that (ref. \cite{Zi81}) for each ideal class $\mathcal{C}\in Cl(K),$ there exists an integral ideal $\mathfrak{a}_{\mathcal{C}}\in \mathcal{C}$ such that
$$
N_{K/{\mathbb{Q}}}(\mathfrak{a}_{\mathcal{C}})\leq M(0,n)\sqrt{D_K}.
$$
Clearly, we may assume $\mathfrak{a}_{[\mathcal{O}_K]}=\mathcal{O}_K.$ Thus we can define a set of representatives of $Cl(K)$ as
\begin{equation}\label{I_K'}
\mathcal{I}_K:=\Big\{\mathfrak{a}:\; \mathfrak{a}=\mathfrak{a}_{\mathcal{C}}\mathfrak{q},\quad \forall\; \mathcal{C}\in Cl(K)\Big\},
\end{equation}
where $\mathfrak{q}$ is the index ideal in \eqref{cm}.

For convenience, let us fix $\mathcal{I}_K$ once for all. Clearly we have
$$
Cl(K)=\{[\mathfrak{a}]:\; \mathfrak{a}\in \mathcal{I}_K\}.
$$
For any $\mathfrak{a}\in \mathcal{I}_K,$ let $y_{\mathfrak{a}}$ be the imaginary part of $z_{\mathfrak{a}},$ the associated CM point. Then by \eqref{cm2} we have
\begin{equation}\label{cm3}
y^{\bm{\sigma}}_{\mathfrak{a}}=N_{\Phi}(y_{\mathfrak{a}})
\geq\frac{M(0,n)^{-1}N_{F/{\mathbb{Q}}}(\mathfrak{f}_{\mathfrak{a}})}{2^nD_F},\quad\forall\; \mathfrak{a}\in \mathcal{I}_K.
\end{equation}

Also, for any fractional $\mathfrak{a}$ of $K,$ there exists a unique $\mathfrak{f}_{\mathfrak{a}}\in \mathcal{I}^{+}_F$ and a CM point $[\mathfrak{a},r]\in X(\mathfrak{f}_{\mathfrak{a}}):=\Gamma(\mathfrak{f}_{\mathfrak{a}})\setminus\mathbb{H}^n$ mapping to $[\mathfrak{a}].$ Note that there are at most 2 pre-images of $[\mathfrak{a}]\in Cl(K).$ From now on, we fix one of them $[\mathfrak{a},r]\in \mathcal{CM}(K,\Phi,\mathfrak{f}_{\mathfrak{a}}),$ $\forall$ $\mathfrak{a}.$

Then we have a decomposition \eqref{cm0}, i.e.
$$
\mathfrak{a}=\mathcal{O}_F\alpha_{\mathfrak{a}}+\mathfrak{f}_{\mathfrak{a}}\beta_{\mathfrak{a}}
$$
with $z_{\mathfrak{a}}:=\frac{\alpha_{\mathfrak{a}}}{\beta_{\mathfrak{a}}}\in K^{\times}\cap\mathbb{H}^n=\{z\in K^{\times}: \Phi(z)\in \mathbb{H}^n\}$. Moreover, $z_{\mathfrak{a}}$ represents the CM point $[\mathfrak{a},r].$

\begin{prop}\label{1}
Let $K$ be a CM extension of a totally real number field $F$ of degree $n,$ and $\Phi$ be a CM type of $K.$ Then we have
\begin{equation}\label{ly}
L_K(\chi,s)=\frac{\left(2^nD_F\right)^s}{D_K^{\frac{s}2}[\mathcal{O}^{\times}_K:\mathcal{O}^{\times}_F]}\sum_{[\mathfrak{a}^{-1}]\in Cl(K)}\overline{\chi}([\mathfrak{a}])N(\mathfrak{f}_{\mathfrak{a}})^{s}\zeta_F(2s)
E\left(z_{\mathfrak{a}},s;\mathfrak{f}^{-1}_{\mathfrak{a}},\mathfrak{f}_{\mathfrak{a}}\right),
\end{equation}
where $\mathfrak{f}_{\mathfrak{a}}\in \mathcal{I}^{+}_F$ is defined as above, $z_{\mathfrak{a}}$ is the corresponding CM points of $\mathfrak{a}$ via the map $\mathcal{CM}(K,\Phi)\twoheadrightarrow Cl(K).$
\end{prop}
\begin{proof}
Let $C\in Cl(K)$ be an ideal class. Then there exist a unique primitive ideal $\mathfrak{a}\in\mathcal{I}_K$ such that $[\mathfrak{a}]=C^{-1}.$ Hence as $\mathfrak{b}$ runs over integral ideals in $C$, $\mathfrak{a}\mathfrak{b}=(w)$ runs over principal ideals $(w)$ with $w\in \mathfrak{a}/\mathcal{O}^{\times}_K.$ Let $\sum'$ denote that the summation is taking over nonzero integral variables (e.g. $\sum_{\mathfrak{a}}'$ means the summation is taken over all nonzero integral ideals $\mathfrak{a}\subset K$), then the partial Dedekind zeta function can be written
\begin{align*}
\zeta_K(s,C)&=\mathop{{\sum}'}_{\mathfrak{b}\in C}N_{K/\mathbb{Q}}(\mathfrak{b})^{-s}=N_{K/\mathbb{Q}}(\mathfrak{a})^{s}\mathop{{\sum}'}_{w\in \mathfrak{a}/\mathcal{O}^{\times}_K}N_{K/\mathbb{Q}}\left((w)\right)^{-s}\\
&=\frac{N_{K/\mathbb{Q}}(\mathfrak{a})^{s}}{[\mathcal{O}^{\times}_K:\mathcal{O}^{\times}_F]}\mathop{{\sum}'}_{(c,d)\in \mathcal{O}_F\oplus\mathfrak{f}_{\mathfrak{a}}/\mathcal{O}^{\times}_F}N_{K/\mathbb{Q}}\left((c\alpha_{\mathfrak{a}}+d\beta_{\mathfrak{a}})\right)^{-s}\\
&=\frac{N_{K/\mathbb{Q}}(\mathfrak{a})^{s}N_{K/\mathbb{Q}}\left((\beta_{\mathfrak{a}})\right)^{-s}}{[\mathcal{O}^{\times}_K:\mathcal{O}^{\times}_F]}
\mathop{{\sum}'}_{(c,d)\in\mathcal{O}_F\oplus\mathfrak{f}_{\mathfrak{a}}/\mathcal{O}^{\times}_F}N_{K/\mathbb{Q}}\left((cz_{\mathfrak{a}}+d)\right)^{-s}\\
&=\frac{N_{K/\mathbb{Q}}(\mathcal{O}_Fz_{\mathfrak{a}}+\mathfrak{f}_{\mathfrak{a}})^{s}}{[\mathcal{O}^{\times}_K:\mathcal{O}^{\times}_F]}
\mathop{{\sum}'}_{(c,d)\in\mathcal{O}_F\oplus\mathfrak{f}_{\mathfrak{a}}/\mathcal{O}^{\times}_F}N_{K/\mathbb{Q}}\left((cz_{\mathfrak{a}}+d)\right)^{-s}.\\
\end{align*}
Write $z_{\mathfrak{a}}=x_{\mathfrak{a}}+iy_{\mathfrak{a}},$ then a calculation with determinants yields
\begin{equation}\label{Yo}
N_{K/\mathbb{Q}}(\mathcal{O}_Fz_{\mathfrak{a}}+\mathfrak{f}_{\mathfrak{a}})
=y_{\mathfrak{a}}^{\bm{\sigma}}N_{F/\mathbb{Q}}(\mathfrak{f}_{\mathfrak{a}})\cdot\frac{2^nD_F}{\sqrt{D_K}}.
\end{equation}
By a calculation with the CM type $\Phi$ we have$
N_{K/\mathbb{Q}}\left((cz_{\mathfrak{a}}+d)\right)=|cz_{\mathfrak{a}}+d|^{2\bm{\sigma}},
$
where we have identified $z_{\mathfrak{a}}$ with $\Phi(z_{\mathfrak{a}})\in\mathbb{H}^n.$
Thus by combining the preceding computations we obtain
\begin{align*}
\zeta_K(s,C)&=\frac{\left(2^nD_F N_{F/\mathbb{Q}}(\mathfrak{f}_{\mathfrak{a}})\right)^s}{D_K^{\frac{s}2}[\mathcal{O}^{\times}_K:\mathcal{O}^{\times}_F]}
\mathop{{\sum}'}_{(c,d)\in\mathcal{O}_F\oplus\mathfrak{f}_{\mathfrak{a}}/\mathcal{O}^{\times}_F}y_{\mathfrak{a}}^{s\bm{\sigma}}|cz_{\mathfrak{a}}
+d|^{-2s\bm{\sigma}}\\
&=\frac{\left(2^nD_F N_{F/\mathbb{Q}}(\mathfrak{f}_{\mathfrak{a}})\right)^s}{D_K^{\frac{s}2}[\mathcal{O}^{\times}_K:\mathcal{O}^{\times}_F]}
G\left(z_{\mathfrak{a}},s;\mathfrak{f}^{-1}_{\mathfrak{a}},\mathfrak{f}_{\mathfrak{a}}\right).
\end{align*}
Finally, using that $L_K(\chi,s)=\sum_{C\in Cl(K)}\chi(C)\zeta_K(s,C)$
to obtain the formula \eqref{ly}.
\end{proof}

In particular, it comes from Lemma \ref{0} and Proposition \ref{1} that 
\begin{prop}\label{11}
	Let notations be as above. Then we have
	\begin{equation}\label{l}
	L_K\left(\chi,\frac12\right)=\frac{2^{\frac{n}2-1}\rho_F\sqrt{D_F}}{D_K^{\frac14}[\mathcal{O}^{\times}_K:\mathcal{O}^{\times}_F]}\sum_{[\mathfrak{a}^{-1}]\in Cl(K)}\overline{\chi}([\mathfrak{a}])\sqrt{N(\mathfrak{f}_{\mathfrak{a}})}
	E'\left(z_{\mathfrak{a}},\frac12;\mathfrak{f}^{-1}_{\mathfrak{a}},\mathfrak{f}_{\mathfrak{a}}\right),
	\end{equation}
\end{prop}
\begin{remark}
	Note that $\sqrt{N_{F/\mathbb{Q}}(\mathfrak{f}_{\mathfrak{a}})}
	E'\left(z_{\mathfrak{a}},\frac12;\mathfrak{f}^{-1}_{\mathfrak{a}},\mathfrak{f}_{\mathfrak{a}}\right)$ is independent of the choice of $\mathfrak{f}_{\mathfrak{a}},$ for any fractional ideal $\mathfrak{a}$ in $K.$ Formula \eqref{l} is known when the narrow class group of $F$ is trivial (ref. \cite{Ma10}), in which case one can take $\mathfrak{f}_{\mathfrak{a}}$ to be $\mathcal{O}_K$.
\end{remark}
We will use this Eisenstein period formula \eqref{l} in conjunction with the CM type norm formula \eqref{cm2}, following an idea of Iwaniec and Kowalski \cite{IK04}, to obtain Theorem \ref{main} in Section \ref{sec4}.

\subsection{Derivatives of Eisenstein Series at the Central Point}
In this subsection the derivative of the Eisenstein series and its Fourier expansion will be investigated. While further estimates will be provided in the next section. We start from the vanishing property of $E'\left(z,\frac12;\mathfrak{a},\mathfrak{b}\right)$ as follows.
\begin{lemma}\label{Der}
Let $\mathfrak{a}$ and $\mathfrak{b}$ be fractional ideals of $F$ as before. Then we have
\begin{align*}
E'\left(z,\frac12;\mathfrak{a},\mathfrak{b}\right)=&\frac{y^{\frac{\bm{\sigma}}2}}{N_{F/\mathbb{Q}}(\mathfrak{b})}\Bigg\{
2h^{-1}_F\log y^{\bm{\sigma}}+\frac{4(\gamma_{F,[\mathfrak{b}]^{-1}}-h^{-1}_F\gamma_F)}{h_F\rho_F}
\Bigg.\\
&\phantom{=\;\;}\Bigg.
+h_F^{-1}\Big[\log N_{F/\mathbb{Q}}(\mathfrak{a}\mathfrak{b}^{-1})-n(\gamma+2\log2)\Big]
\Bigg.\\
&\phantom{=\;\;}\Bigg.
+\frac{2^{n+1}}{\rho_F\sqrt{D_F}}\sum_{b\in F^{\times}}\mathop{\sum\sum}_{\substack{(a,c)\in\mathfrak{b}^{-1}\mathfrak{o}^{-1}\times\mathcal{C}\\ac=b}}\mathbf{e}(bx)\prod_{v\in\mathbf{J}_{\infty}}K_0\left(2\pi y_v|b_v|\right)
\Bigg\},
\end{align*}
where the Euler-Kronecker constants $\gamma_{F,[\mathfrak{b}]^{-1}}$ and $\gamma_F$ are defined in \eqref{zeta} and \eqref{zeta1} respectively; $\gamma=0.57721\dots$ is the Euler-Mascheroni constant and
$$
\mathcal{C}:=\left(\mathfrak{a}\mathfrak{b}\cap F^{\times}\right)/\mathcal{O}^{\times,+}_F.
$$
\end{lemma}

\begin{proof}
To simplify the computation, let's introduce some notation. Set
\begin{align*}
&M_1(s):=N_{F/\mathbb{Q}}(\mathfrak{b})^{-2s}y^{s\bm{\sigma}}\frac{\zeta_F(2s,[\mathfrak{b}]^{-1})}{\zeta_F(2s)};\\
&M_2(s):=\left(\frac{\sqrt{\pi}\Gamma(s-1/2)}{\Gamma(s)}\right)^nD^{-\frac12}_FN(\mathfrak{b})^{-1}y^{(1-s)\bm{\sigma}}N(\mathfrak{a}\mathfrak{b})^{1-2s}
\frac{\zeta_F(2s-1,[\mathfrak{a}\mathfrak{b}]^{-1})}{\zeta_F(2s)};\\
&E(s):=\frac{2^n\pi^{ns}y^{\frac{\bm{\sigma}}2}}{\sqrt{D_F}\Gamma(s)^nN(\mathfrak{b})\zeta_F(2s)}\sum_{b\in F^{\times}}|N(b)|^{s-\frac12}\lambda(b,s)\mathbf{e}(bx)\prod_{v\in\mathbf{J}_{\infty}}K_{s-\frac12}\left(2\pi y_v|b_v|\right).
\end{align*}
Note that for convenience, we short the notation of the norm $N_{F/\mathbb{Q}}$ just for $N$ occasionally. Then clearly $E'(z,1/2;\mathfrak{a},\mathfrak{b})=M'_1(1/2)+M'_2(1/2)+E'(1/2).$ Also,
\begin{align*}
M'_1(s)=\Big[-2\log N(\mathfrak{b})+\log y^{\bm{\sigma}}\Big]\cdot M_1(s)+N(\mathfrak{b})^{-2s}y^{s\bm{\sigma}}\left(\frac{\zeta_F(2s,[\mathfrak{b}]^{-1})}{\zeta_F(2s)}\right)'.
\end{align*}
By \eqref{zeta} and \eqref{zeta1} we have
\begin{align*}
\left(\frac{\zeta_F(2s,[\mathfrak{b}]^{-1})}{\zeta_F(2s)}\right)'\mid_{s=\frac12}
&=\lim_{s\mapsto\frac12}\frac{2\zeta'_F(2s,[\mathfrak{b}]^{-1})\zeta_F(2s)-2\zeta'_F(2s)\zeta_F(2s,[\mathfrak{b}]^{-1})}{\zeta^2_F(2s)}\\
&=\lim_{s\mapsto\frac12}\frac{\frac{4\rho_F}{(2s-1)^2}
\left(\frac{h^{-1}_F\rho_F}{2s-1}+\gamma_{F,[\mathfrak{b}]^{-1}}\right)-\frac{4h^{-1}_F\rho_F}{(2s-1)^2}\left(\frac{\rho_F}{2s-1}+\gamma_{F}\right)}
{\frac{\rho_F^2}{(2s-1)^2}}\\
&=\frac{4(\gamma_{F,[\mathfrak{b}]^{-1}}-h^{-1}_F\gamma_F)}{\rho_F}.
\end{align*}
Note that by \eqref{zeta} and \eqref{zeta1} we have $
M_1\left(\frac12\right)=\frac{y^{\frac{\bm{\sigma}}2}}{N(\mathfrak{b})}\lim_{s\mapsto\frac12}\frac{\zeta_F(2s,[\mathfrak{b}]^{-1})}{\zeta_F(2s)}
=\frac{y^{\frac{\bm{\sigma}}2}}{h_FN(\mathfrak{b})}.$
Thus we have
\begin{equation}\label{M_1}
M'_1(1/2)=\frac{y^{\frac{\bm{\sigma}}2}}{h_FN(\mathfrak{b})}\Bigg\{
\log\frac{y^{\bm{\sigma}}}{N(\mathfrak{b})^2}+\frac{4(\gamma_{F,[\mathfrak{b}]^{-1}}-h^{-1}_F\gamma_F)}{\rho_F}\Bigg\}.
\end{equation}
For the $M'_2(1/2)-$term, by definition we have
\begin{align*}
\log M_2(s)&=C+n\log\Gamma(s-1/2)-n\log\Gamma(s)+(1-s)\log y^{\bm{\sigma}}\\
&\quad+(1-2s)\log N(\mathfrak{a}\mathfrak{b})+\log\zeta_F(2s-1,[\mathfrak{a}\mathfrak{b}]^{-1})-\log\zeta_F(2s),
\end{align*}
where $C:=n\log\sqrt{\pi}-1/2\log D_F-\log N(\mathfrak{b}).$ From this identity we obtain
\begin{align*}
M'_2(s)=M_2(s)\cdot&\Bigg\{\frac{n\Gamma'(s-\frac12)}{\Gamma(s-\frac12)}-\frac{n\Gamma'(s)}{\Gamma(s)}-\log y^{\bm{\sigma}}-2s\log N(\mathfrak{a}\mathfrak{b})
\Bigg.\\
&\phantom{=\;\;}\Bigg.
+\frac{2\zeta'_F(2s-1,[\mathfrak{a}\mathfrak{b}]^{-1})}{\zeta_F(2s-1,[\mathfrak{a}\mathfrak{b}]^{-1})}
-\frac{2\zeta'_F(2s-1)}{\zeta_F(2s-1)}\Bigg\}.
\end{align*}
Since $\Gamma(s-1/2)\sim (s-1/2)^{-1}$ around $s=1/2$ and noting \eqref{zeta1} and \eqref{zeta2}, we obtain
\begin{align*}
M'_2(1/2)&=M_2(1/2)\lim_{s\mapsto1/2}\Bigg\{-\frac{n}{s-1/2}-\frac{n\Gamma'(1/2)}{\Gamma(1/2)}-\log y^{\bm{\sigma}}
\Bigg.\\
&\qquad\phantom{=\;\;}\Bigg.
-2s\log N(\mathfrak{a}\mathfrak{b})+\frac{n-1}{s-1/2}+\frac{1}{s-1/2}\Bigg\}\\
&=-M_2(1/2)\cdot\Bigg\{\frac{n\Gamma'(1/2)}{\Gamma(1/2)}+\log y^{\bm{\sigma}}+\log N(\mathfrak{a}\mathfrak{b})\Bigg\}.
\end{align*}
By \eqref{zeta1} and \eqref{zeta2} and the functional equation \eqref{2.2.4} one can easily deduce that
$
M_2(1/2)=-h_F^{-1}N(\mathfrak{b})^{-1}y^{\bm{\sigma}/2}.$
Hence we have
\begin{equation}\label{M_2}
M'_2(1/2)=h_F^{-1}N(\mathfrak{b})^{-1}y^{\bm{\sigma}/2}\cdot\Bigg\{\frac{n\Gamma'(1/2)}{\sqrt{\pi}}+\log y^{\bm{\sigma}}+\log N(\mathfrak{a}\mathfrak{b})\Bigg\}.
\end{equation}
Now let's compute $\Gamma'(1/2)$: Differentiating the Hadamard decomposition of $\Gamma(s)^{-1}$ logarithmically at $s=1$ we see $\Gamma'(1)=-\gamma$ by the definition of $\gamma.$ Since $\Gamma(s+1)=s\Gamma(s),$ we have $\Gamma'(2)=1-\gamma.$ Now consider the duplication formula $\Gamma(2s)=\pi^{-1/2}2^{2s-1}\Gamma(s)\Gamma(s+1/2).$
Differentiating it at $s=1/2$ we thus obtain $\Gamma'(1/2)=-\sqrt{\pi}(\gamma+2\log 2).$ Plug this into \eqref{M_2} to obtain
\begin{equation}\label{M_2'}
M'_2(1/2)=h_F^{-1}N(\mathfrak{b})^{-1}y^{\bm{\sigma}/2}\cdot\{\log y^{\bm{\sigma}}+\log N(\mathfrak{a}\mathfrak{b})-n(\gamma+2\log 2)\}.
\end{equation}
Finally we deal with $E'(1/2)-$term. Noting that $\lim_{s\mapsto1/2}\zeta^{-1}_F(2s)=0,$ we have
\begin{align*}
E'\left(\frac12\right)&=\frac{-2^ny^{\frac{\bm{\sigma}}2}}{\sqrt{D_F}N(\mathfrak{b})}\lim_{s\mapsto\frac12}\frac{2\zeta'_F(2s)}{\zeta^2_F(2s)}\sum_{b\in F^{\times}}|N(b)|^{s-\frac12}\lambda(b,s)\mathbf{e}(bx)\prod_{v\in\mathbf{J}_{\infty}}K_{s-\frac12}\left(2\pi y_v|b_v|\right)\\
&=\frac{2^{n+1}y^{\frac{\bm{\sigma}}2}}{D^{1/2}_FN(\mathfrak{b})\rho_F}\sum_{b\in F^{\times}}\lambda\left(b,\frac12\right)\mathbf{e}(bx)\prod_{v\in\mathbf{J}_{\infty}}K_{s-\frac12}\left(2\pi y_v|b_v|\right)\\
&=\frac{2^{n+1}y^{\frac{\bm{\sigma}}2}}{D^{1/2}_FN(\mathfrak{b})\rho_F}\sum_{b\in F^{\times}}\mathop{\sum\sum}_{\substack{(a,c)\in\mathfrak{b}^{-1}\mathfrak{o}^{-1}\times\mathcal{C}\\ac=b}}\mathbf{e}(bx)\prod_{v\in\mathbf{J}_{\infty}}
K_0\left(2\pi y_v|b_v|\right).
\end{align*}
Combining this formula with \eqref{M_1} and \eqref{M_2'} we thus obtain the conclusion.
\end{proof}

\section{Proof of the Main Theorems}\label{sec4}
\subsection{Estimates Related to L-functions}\label{sec4.1}
Let $F$ be a totally real number field of degree $n.$ Let $K/F$ be a CM extension and $\widehat{Cl(K)}$ be the dual group of the ideal class group $Cl(K).$ Note that $L(\chi,1)$ is finite for any nontrivial $\chi\in\widehat{Cl(K)},$ so we can define
\begin{equation}\label{gamma}
\mathcal{L}_{F}:=\max_{\chi\in\widehat{Cl(F)}\setminus\{\chi_0\}}|L_F(\chi,1)|.
\end{equation}
Also, in this paper we will alway use $M(r_1, r_2)$ as the generalized Minkowski function defined in \eqref{zi*}. 
\begin{prop}\label{fin}
Let notation be as above, and $\chi_0$ be the trivial character in $\widehat{Cl(K)}.$ Then we have
\begin{equation}\label{4.2}
L_K\left(\chi_0,\frac12\right)\geq \frac{\rho_F}{[\mathcal{O}^{\times}_K:\mathcal{O}^{\times}_F]\cdot h_F}\left(\frac12
\log\frac{\sqrt{D_K}}{D_F}-\Phi^0_F\cdot h_KD_K^{-\frac14}\right).
\end{equation}
where
$$
\Phi^0_F:=\frac{2^{5n/2}M(0,n)D_F^{7/4}h_F^2}{\pi^n\rho_Fh^{+}_F}+e^{2\rho_F^{-1}\mathcal{L}_{F}+n}\mathcal{L}_{F}^*\sqrt{D_F},
$$
and $\mathcal{L}_{F}^*=4\rho_F^{-1}\mathcal{L}_{F}+\log D_F+(3\log2-\log\pi)n+\sqrt{7n}+4.$
\end{prop}
\begin{proof}
By Lemma \ref{Der} we have $E'\left(z,\frac12;\mathfrak{f}_{\mathfrak{a}}^{-1},\mathfrak{f}_{\mathfrak{a}}\right)=I_{M}(z;\mathfrak{f}_{\mathfrak{a}})+I_{E}(z;\mathfrak{f}_{\mathfrak{a}}),$
where
\begin{align*}
I_{M}(z;\mathfrak{f}_{\mathfrak{a}})&:=\frac{y^{\frac{\bm{\sigma}}2}}{N_{F/\mathbb{Q}}(\mathfrak{f}_{\mathfrak{a}})h_F}\Bigg\{
2\log{y^{\bm{\sigma}}}+\frac{4\Upsilon_{F,[\mathfrak{f}_{\mathfrak{a}}]^{-1}}}{\rho_F}
-2\log N_{F/\mathbb{Q}}(\mathfrak{f}_{\mathfrak{a}})-(\gamma+2\log2)n\Bigg\},\\
I_{E}(z;\mathfrak{f}_{\mathfrak{a}})&:=\frac{y^{\frac{\bm{\sigma}}2}}{N_{F/\mathbb{Q}}(\mathfrak{f}_{\mathfrak{a}})}\cdot\frac{2^{n+1}}{\rho_F\sqrt{D_F}}\sum_{b\in F^{\times}}\mathop{\sum\sum}_{\substack{(a,c)\in\mathfrak{f}_{\mathfrak{a}}^{-1}\mathfrak{o}^{-1}\times\mathcal{C}\\ac=b}}\mathbf{e}(bx)
\prod_{v\in\mathbf{J}_{\infty}}K_0\left(2\pi y_v|b_v|\right),
\end{align*}
where $\mathcal{C}:=\mathcal{O}_F^{\times}/\mathcal{O}_F^{\times,+},$ and
$$
\mathcal{L}_{F,[\mathfrak{f}_{\mathfrak{a}}]^{-1}}:=\gamma_{F,[\mathfrak{f}_{\mathfrak{a}}]^{-1}}-h^{-1}_F\gamma_F=\frac1{h_F}\sum_{\substack{\chi\in\widehat{Cl(F)}
\\ \chi\neq\chi_0}}\chi([\mathfrak{f}_{\mathfrak{a}}])L_F(\chi,1).
$$
By Proposition \ref{1} we can write $L_K(\chi,1/2)=L_{M,\chi}+L_{E,\chi},$ where
\begin{align*}
L_{M,\chi}&=\frac{2^{\frac{n}2}\sqrt{D_F}}{2D_K^{\frac14}}\cdot\frac{\rho_F}{[\mathcal{O}^{\times}_K:\mathcal{O}^{\times}_F]}\sum_{[\mathfrak{a}^{-1}]\in Cl(K)}\overline{\chi}([\mathfrak{a}])\sqrt{N_{F/\mathbb{Q}}(\mathfrak{f}_{\mathfrak{a}})}\cdot I_{M}(z;\mathfrak{f}_{\mathfrak{a}}),\\
L_{E,\chi}&=\frac{2^{\frac{n}2}\sqrt{D_F}}{2D_K^{\frac14}}\cdot\frac{\rho_F}{[\mathcal{O}^{\times}_K:\mathcal{O}^{\times}_F]}\sum_{[\mathfrak{a}^{-1}]\in Cl(K)}\overline{\chi}([\mathfrak{a}])\sqrt{N_{F/\mathbb{Q}}(\mathfrak{f}_{\mathfrak{a}})}\cdot I_{E}(z;\mathfrak{f}_{\mathfrak{a}}).
\end{align*}
We will start with bounding $I_{E}(z;\mathfrak{f}_{\mathfrak{a}})$ and further estimating $L_{E,\chi}:$
\begin{align*}
\Big|I_{E}(z;\mathfrak{f}_{\mathfrak{a}})\Big|&\leq\frac{y^{\frac{\bm{\sigma}}2}}{N(\mathfrak{f}_{\mathfrak{a}})}\cdot\frac{2^{n+1}}{\rho_F\sqrt{D_F}}\sum_{b\in F^{\times}}\mathop{\sum\sum}_{\substack{(a,c)\in\mathfrak{f}_{\mathfrak{a}}^{-1}\mathfrak{o}^{-1}\times\mathcal{C}\\ac=b}}
\prod_{v\in\mathbf{J}_{\infty}}\Big|K_0\left(2\pi y_v|b_v|\right)\Big|\\
&=\frac{y^{\frac{\bm{\sigma}}2}}{N(\mathfrak{f}_{\mathfrak{a}})}\cdot\frac{2^{n+1}[\mathcal{O}_F^{\times}:\mathcal{O}_F^{\times,+}]}{\rho_F\sqrt{D_F}}\sum_{b\in \mathfrak{f}_{\mathfrak{a}}^{-1}\mathfrak{o}^{-1}\cap F^{\times}}\prod_{v\in\mathbf{J}_{\infty}}\Big|K_0\left(2\pi y_v|b_v|\right)\Big|\\
\end{align*}
To compute $[\mathcal{O}_F^{\times}:\mathcal{O}_F^{\times,+}],$ let's fix an ordering $(\phi_1,\cdots,\phi_n)$ of $Hom(F,\overline{\mathbb{Q}})$ and consider the homomorphism
\begin{align*}
\tau:\; \mathcal{O}_F^{\times}\longrightarrow\{\pm1\}^n,\quad x\mapsto(\phi_1(x)/|\phi_1(x)|,\cdots,\phi_n(x)/|\phi_n(x)|).
\end{align*}
Then clearly, $\ker(\tau)=\mathcal{O}_F^{\times,+}.$ Hence $\Im(\tau)\simeq\mathcal{O}_F^{\times}/\mathcal{O}_F^{\times,+}.$ In fact, by Lemma 11.2 of \cite{CH88} we have
$$
\coker(\tau)\simeq \Gal(H^{+}_F/H_F)\simeq \ker(Cl(F)^{+}\rightarrow Cl(F)),
$$
where $H_F$ is the Hilbert class field of $F$ and $H^{+}_F$ is the narrow Hilbert class field of $F.$ By the above isomorphism, we have $[\mathcal{O}_F^{\times}:\mathcal{O}_F^{\times,+}]={2^nh_F}{h^{+}_F}^{-1}.$
Noting that $F$ is totally real, consider the canonical embedding
$$
j:\; F\longrightarrow F_{\mathbb{R}}:=\prod_{v\in\mathbf{J}_{\infty}}F_v\simeq\mathbb{R}^n.
$$
Since $\mathfrak{f}_\mathfrak{a}^{-1}\mathfrak{o}^{-1}\neq0,$ the lattice $\Gamma:=j(\mathfrak{f}_\mathfrak{a}^{-1}\mathfrak{o}^{-1})$ is complete in $F_{\mathbb{R}}.$ Let $\alpha_1,\cdots,\alpha_n$ be a $\mathbb{Z}-$basis of $\mathfrak{f}_\mathfrak{a}^{-1}\mathfrak{o}^{-1}.$ Let $\beta_v=j(\alpha_v),$ then we may assume that $\beta_v>0,$ $\forall$ $1\leq v\leq n.$ Since $\mathbb{Z}$ is a PID, we have $\Gamma=\mathbb{Z}\beta_1\oplus\cdots\oplus\mathbb{Z}\beta_n.$
Then a computation with determinants gives
\begin{align*}
\prod_{v\in\mathbb{J}_{\infty}}\beta_v^{-1}=\vol \left( \Gamma\right) ^{-1}=\frac1{\sqrt{D_F}}\cdot N_{F/{\mathbb{Q}}}(\mathfrak{f}_{\mathfrak{a}}\mathfrak{o})= N_{F/{\mathbb{Q}}}(\mathfrak{f}_{\mathfrak{a}})\sqrt{D_F}.
\end{align*}
For any $b\in\mathfrak{f}_{\mathfrak{a}}^{-1}\mathfrak{o}^{-1}\cap F^{\times},$ we may write $
j(b)=(m_1\beta_1,\cdots,m_n\beta_n),$ where $m_i\neq0,$ $\forall$ $1\leq v\leq n.$
Otherwise, we may assume $m_1=0.$ Then there exists some $v\in \mathbf{J}_{\infty}$ such that $b_v=0.$ Then the minimal polynomial of $b$ has 0 as its root. This is impossible unless $b=0.$

On the other hand, note that $K_0(x)<K_{1/2}(x)=\sqrt{\frac{\pi}{2x}}e^{-x},$ $\forall$ $x>0.$ Combining these results we have
\begin{align*}
\Big|I_{E}(z;\mathfrak{f}_{\mathfrak{a}})\Big|&\leq\frac{y^{\frac{\bm{\sigma}}2}}{N_{F/\mathbb{Q}}(\mathfrak{f}_{\mathfrak{a}})}\cdot\frac{2^{2n+1}h_F}
{\rho_Fh^{+}_F\sqrt{D_F}}\prod_{v\in\mathbf{J}_{\infty}}\left(\sum_{m=1}^{\infty}\sqrt{\frac1{y_v\beta_vm}}e^{-2\pi y_v\beta_vm}\right)\\
&=\frac{1}{N_{F/\mathbb{Q}}(\mathfrak{f}_{\mathfrak{a}})}\cdot\frac{2^{2n+1}h_F}{\rho_Fh^{+}_F\sqrt{D_F}}\prod_{v\in\mathbf{J}_{\infty}}\beta_v^{-1/2}
\prod_{v\in\mathbf{J}_{\infty}}\left(\sum_{m=1}^{\infty}\sqrt{\frac1{m}}e^{-2\pi y_v\beta_vm}\right)\\
&\leq\frac{1}{N_{F/\mathbb{Q}}(\mathfrak{f}_{\mathfrak{a}})}\cdot\frac{2^{2n+1}h_F}{\rho_Fh^{+}_F\sqrt{D_F}}\prod_{v\in\mathbf{J}_{\infty}}\beta_v^{-1/2}
\prod_{v\in\mathbf{J}_{\infty}}\frac1{e^{2\pi y_v\beta_v}-1}\\
&\leq\frac{1}{N_{F/\mathbb{Q}}(\mathfrak{f}_{\mathfrak{a}})}\cdot\frac{2^{2n+1}h_F}{\rho_Fh^{+}_F\sqrt{D_F}}\prod_{v\in\mathbf{J}_{\infty}}\beta_v^{-1/2}
\prod_{v\in\mathbf{J}_{\infty}}\frac1{2\pi y_v\beta_v}\\
&\leq \sqrt{N_{F/\mathbb{Q}}(\mathfrak{f}_{\mathfrak{a}})}\cdot\frac{2^{n+1}h_FD_F^{1/4}}{\pi^n\rho_Fh^{+}_F}y^{-\bm{\sigma}}.\\
\end{align*}
Therefore, by Lemma \ref{prim} and the definition of $\mathcal{I}_K,$ we have
\begin{align*}
\Big|I_{E}(z;\mathfrak{f}_{\mathfrak{a}})\Big|&\leq \sqrt{N_{F/\mathbb{Q}}(\mathfrak{f}_{\mathfrak{a}})}\cdot\frac{2^{n+1}h_FD_F^{1/4}}{\pi^n\rho_Fh^{+}_F}\cdot
\frac{2^nD_FN_{K/{\mathbb{Q}}}(\mathfrak{a})}{N_{F/{\mathbb{Q}}}(\mathfrak{f}_{\mathfrak{a}})N_{F/{\mathbb{Q}}}(\mathfrak{q})^2\sqrt{D_K}}\\
&\leq\frac{2^{2n+1} M(0,n)h_F{D_F}^{5/4}}{\pi^n\rho_Fh^{+}_F\sqrt{N_{F/\mathbb{Q}}(\mathfrak{f}_{\mathfrak{a}})}}, \quad\forall\; \mathfrak{a}\in\mathcal{I}_K.
\end{align*}
Note that by definition, each $\mathfrak{f}_{\mathfrak{a}}\in \mathcal{I}^{+}_F$ defined in \eqref{I_F}. Thus we have
\begin{align*}
\Big|L_{E,\chi}\Big|&\leq\frac{2^{\frac{n}2}\sqrt{D_F}}{2D_K^{\frac14}}\cdot\frac{\rho_F}{[\mathcal{O}^{\times}_K:\mathcal{O}^{\times}_F]}\sum_{\mathfrak{a}\in \mathcal{I}_K}\frac{2^{2n+1} M(0,n)h_F{D_F}^{5/4}}{\pi^n\rho_Fh^{+}_F}\\
&\leq\frac1{[\mathcal{O}^{\times}_K:\mathcal{O}^{\times}_F]}\cdot \frac{2^{5n/2}M(0,n)D_F^{7/4}h_F}{\pi^nh^{+}_F}\cdot h_KD_K^{-1/4},
\end{align*}
where $\mathcal{I}_K$ is defined in \eqref{I_K'}, so that \eqref{cm3} is available here.

On the other hand, we will give a lower bound for $I_{M}(z;\mathfrak{f}_{\mathfrak{a}})$ and $L_{M,\chi}$. Recall the definition of $\mathcal{L}_{F}$ given in \eqref{gamma}, then clearly for any $\mathfrak{a}\in\mathcal{I}_K,$ one has
$$
\mathcal{L}_{F}\geq\frac1{h_F}\sum_{\chi\in\widehat{Cl(F)}\setminus\{\chi_0\}}\Big|L_F(\chi,1)\Big|\geq \Big|\mathcal{L}_{F,[\mathfrak{f}_{\mathfrak{a}}]^{-1}}\Big|.
$$
Hence by the expression for $y^{\bm{\sigma}}$ given in Lemma \ref{prim} we obtain
\begin{equation}\label{4.1}
I_{M}(z;\mathfrak{f}_{\mathfrak{a}})\geq\sqrt{\frac{N_{K/{\mathbb{Q}}}(c_{\mathfrak{a}})N_{F/{\mathbb{Q}}}(\mathfrak{q})^2}
{N_{K/{\mathbb{Q}}}(\mathfrak{a})N_{F/{\mathbb{Q}}}(\mathfrak{f}_{\mathfrak{a}})}}\Bigg\{
2\log\frac{\sqrt{D_K}}{2^nD_F}-C_{F,{\mathfrak{a}}}\Bigg\}\cdot\frac{D_K^{1/4}}{\sqrt{2^nD_F}h_F},
\end{equation}
where the tail $C_{F,{\mathfrak{a}}}$ is defined as
$$
C_{F,{\mathfrak{a}}}=2\log\frac{N_{K/{\mathbb{Q}}}(\mathfrak{a})}{N_{K/{\mathbb{Q}}}(c_{\mathfrak{a}})N_{F/{\mathbb{Q}}}(\mathfrak{q})^2}
+\frac{4\mathcal{L}_{F}}{\rho_F}
+(\gamma+2\log2)n.
$$
Combining \eqref{4.1}  with \eqref{cm3} yields
\begin{align*}
L_{M,\chi_0}&\geq\frac{\rho_F}{2[\mathcal{O}^{\times}_K:\mathcal{O}^{\times}_F]\cdot h_F}\sum_{[\mathfrak{a}^{-1}]\in Cl(K)}\sqrt{\frac{N_{K/{\mathbb{Q}}}(c_{\mathfrak{a}})N_{F/{\mathbb{Q}}}(\mathfrak{q})^2}
	{N_{K/{\mathbb{Q}}}(\mathfrak{a})}}\\
&\qquad\times \Bigg\{
2\log\frac{\sqrt{D_K}}{2^nD_F}-2\log\frac{N_{K/{\mathbb{Q}}}(\mathfrak{a})}{N_{K/{\mathbb{Q}}}(c_{\mathfrak{a}})N_{F/{\mathbb{Q}}}(\mathfrak{q})^2}
-\frac{4\mathcal{L}_{F}}{\rho_F}
-(\gamma+2\log2)n\Bigg\}\\
&=\frac12\sum_{\mathfrak{a}\in\mathcal{I}_K}\mathcal{N}_{\mathfrak{q}}(\mathfrak{a})\cdot \Bigg\{
2\log\frac{\sqrt{D_K}}{D_F}-2\log\frac{2^nN_{K/{\mathbb{Q}}}(\mathfrak{a})}{N_{K/{\mathbb{Q}}}(c_{\mathfrak{a}})N_{F/{\mathbb{Q}}}(\mathfrak{q})^2}
-C_F\Bigg\},
\end{align*}
where $\mathcal{N}_{\mathfrak{q}}(\mathfrak{a})=\frac{\rho_F\cdot N_{F/{\mathbb{Q}}}(\mathfrak{q})}{[\mathcal{O}^{\times}_K:\mathcal{O}^{\times}_F]\cdot h_F}\cdot \sqrt{\frac{N_{K/{\mathbb{Q}}}(c_{\mathfrak{a}})}
	{N_{K/{\mathbb{Q}}}(\mathfrak{a})}},$ and $C_{F}=4\rho_F^{-1}\mathcal{L}_{F}+(\gamma+2\log2)n.$

Write $\mathcal{N}_{\mathfrak{q}}(\mathfrak{a})^*=2^nN_{K/{\mathbb{Q}}}(c_{\mathfrak{a}})^{-1}N_{F/{\mathbb{Q}}}(\mathfrak{q})^{-2}N_{K/{\mathbb{Q}}}(\mathfrak{a}),$ then we can introduce an undetermined parameter $T$ satisfying $0<T\leq M(0,n)\sqrt{D_K}$ such that 
\begin{align*}
L_{M,\chi_0}&\geq \sum_{\substack{\mathfrak{a}\in\mathcal{I}_K\\ \mathcal{N}_{\mathfrak{q}}(\mathfrak{a})^*\leq T}}\mathcal{N}_{\mathfrak{q}}(\mathfrak{a})\cdot\left(
\log\frac{\sqrt{D_K}}{D_F\cdot \mathcal{N}_{\mathfrak{q}}(\mathfrak{a})^*}-\frac12 C_F\right)-L_{M,\chi_0}^T,
\end{align*}
where the truncation term $L_{M,\chi_0}^T$ is defined as
$$
L_{M,\chi_0}^T= \sum_{\substack{\mathfrak{a}\in\mathcal{I}_K\\ \mathcal{N}_{\mathfrak{q}}(\mathfrak{a})^*\geq T}}\mathcal{N}_{\mathfrak{q}}(\mathfrak{a})\cdot \Bigg\{
\Big|\log\frac{\sqrt{D_K}}{D_F\cdot \mathcal{N}_{\mathfrak{q}}(\mathfrak{a})^*}\Big|
+\frac12C_F\Bigg\}.
$$
We can take $T=
T_{K/F}=\min\Big\{\frac{e^{-C_F}\sqrt{D_K}}{D_F},\;  2^nM(0,n)\sqrt{D_K}\Big\}.$
Then due to \eqref{zi*} and \eqref{I_F} one has $T_{K/F}=e^{-C_F}D_F^{-1}\sqrt{D_K}$ if $n\geq9.$ Note that the choice of $T_{K/F}$ above implies that 
\begin{align*}
\log\frac{\sqrt{D_K}}{D_F\cdot \mathcal{N}_{\mathfrak{q}}(\mathfrak{a})^*}\geq C_F,\quad \text{$\forall$ $\mathfrak{a}\in\mathcal{I}_K$ such that $\mathcal{N}_{\mathfrak{q}}(\mathfrak{a})^*\leq T_{K/F}$.}
\end{align*}
Noting that $\mathcal{N}_{\mathfrak{q}}(\mathfrak{a})^*\leq M(0,n)\sqrt{D_K},$ one then define
$$
\Phi_F=\max_{T_{K/F}\leq\mathcal{N}_{\mathfrak{q}}(\mathfrak{a})^*\leq 2^nM(0,n)\sqrt{D_K}}\Big|\log\frac{\sqrt{D_K}}{D_F\cdot \mathcal{N}_{\mathfrak{q}}(\mathfrak{a})^*}\Big|+\frac12C_F.
$$
Then by monotonicity we obtain 
\begin{equation}\label{phi}
\Phi_F=\max\Big\{\frac32C_F,\ \log 2^nM(0,n)D_F+\frac12C_F\Big\}\leq \log M(0,n)D_F+2C_F.
\end{equation}
When $e^{-C_F}D_F^{-1}\sqrt{D_K}>2^nM(0,n)\sqrt{D_K}$ (e.g. when $n\geq9$), this implies that
\begin{align*}
L_{M,\chi_0}^{T_{K/F}}
&\leq\frac{2^{n/2}\rho_F}{[\mathcal{O}^{\times}_K:\mathcal{O}^{\times}_F]\cdot h_F}\sum_{\substack{\mathfrak{a}\in\mathcal{I}_K\\ T_{K/F}\leq\mathcal{N}_{\mathfrak{q}}(\mathfrak{a})^*\leq 2^nM(0,n)\sqrt{D_K}}}\frac1{\sqrt{\mathcal{N}_{\mathfrak{q}}(\mathfrak{a})^*}}\cdot\Phi_F\\
&=\frac{\rho_F\Phi_Fe^{C_F/2}}{[\mathcal{O}^{\times}_K:\mathcal{O}^{\times}_F]\cdot h_F}\sqrt{D_F}\cdot h_KD_K^{-\frac14}.
\end{align*}
And when $e^{-C_F}D_F^{-1}\sqrt{D_K}\leq2^nM(0,n)\sqrt{D_K}$ (according to our discussion before, this might happen only when $n\leq8$), we just take $L_{M,\chi_0}^{T_{K/F}}=0.$ Hence, we have
\begin{align*}
L_{M,\chi_0}&\geq\frac12\sum_{\substack{\mathfrak{a}\in\mathcal{I}_K\\ \mathcal{N}_{\mathfrak{q}}(\mathfrak{a})^*\leq T_{K/F}}}\mathcal{N}_{\mathfrak{q}}(\mathfrak{a})\cdot
\log\frac{\sqrt{D_K}}{D_F\cdot \mathcal{N}_{\mathfrak{q}}(\mathfrak{a})^*}-L_{M,\chi_0}^{T_{K/F}}\\
&\geq \frac12\mathcal{N}_{\mathfrak{q}}(\mathfrak{q})\cdot
\log\frac{\sqrt{D_K}}{D_F}-\frac{\rho_F\Phi_Fe^{C_F/2}}{[\mathcal{O}^{\times}_K:\mathcal{O}^{\times}_F]\cdot h_F}\sqrt{D_F}\cdot h_KD_K^{-\frac14}.
\end{align*}
Since $\mathcal{N}_{\mathfrak{q}}(\mathfrak{q})={[\mathcal{O}^{\times}_K:\mathcal{O}^{\times}_F]\cdot\rho_F h_F}^{-1},$ then involving the upper bound of $|L_{E,\chi}|$ developed as before we have
\begin{equation}\label{z}
L_K\left(\chi_0,\frac12\right)\geq \frac{\rho_F}{[\mathcal{O}^{\times}_K:\mathcal{O}^{\times}_F]\cdot h_F}\left(\frac12
\log\frac{\sqrt{D_K}}{D_F}-\Phi'_F\cdot h_KD_K^{-1/4}\right).
\end{equation}
where
$$
\Phi'_F:=\frac{2^{5n/2}M(0,n)D_F^{7/4}h_F^2}{\pi^n\rho_Fh^{+}_F}+\Phi_Fe^{C_F/2}\sqrt{D_F}.
$$
By \eqref{zi*} and the usual Minkowski constant $M(n)$ one can obtain an elementary computation of $M(0,n),$ substituting this bound into in \eqref{phi} leads to the inequality
$$
\Phi_F\leq 2C_F+\log e^{-(2\gamma+\log 2\pi)n+\sqrt{7n}+4}D_F,
$$
from which one then has
\begin{equation}\label{y}
\Phi_F\leq 4\rho_F^{-1}\mathcal{L}_{F}+\log D_F+(3\log2-\log\pi)n+\sqrt{7n}+4.
\end{equation}
Then the proof follows from the estimate \eqref{z} and inequalities \eqref{y}.
\end{proof}

From Proposition \ref{fin} one sees naturally that we need an upper bound for $|L_F(\chi,1)|$ to make the inequality \eqref{4.2} more explicit. The desired estimate is provided in the following lemma.
\begin{lemma}\label{average}
Let $F/{\mathbb{Q}}$ be any field extention of degree $n<+\infty.$ Let $\chi$ be any nontrivial primitive Grossencharacter of modulus $\mathfrak{m}.$ Then we have
\begin{equation}\label{3.45}
\Big|L_F(\chi,1)\Big|\leq 2\Big[\frac{e}{2n}\log \left(D_FN_{F/{\mathbb{Q}}}(\mathfrak{m})\right)\Big]^n.
\end{equation}
Moreover, if for any $\chi\in\widehat{Cl(F)}\setminus\{\chi_0\},$ we have
\begin{equation}\label{3.5}
\Big|L_F(\chi,1)\Big|\leq \left(1+\frac{\gamma+2\log2-\log\pi}2n+\frac12\log D_F
+\rho_F^{-1}\gamma_F\right)\cdot\rho_F,
\end{equation}
where $\rho_F:=Res_{s=1}\zeta_F(s)$ and $\gamma_F$ is the Euler-Kronecker constant of $F/{\mathbb{Q}}.$
\end{lemma}
\begin{proof}
Denote by $(r_1, r_2)$ the signature of $F/{\mathbb{Q}}.$ Let $b$ be the number of real places of $F$ dividing the infinite part of the conductor of $\chi.$ Let $a:=r_1-b.$ Then $a\geq0$ and $b\geq0.$
Consider the completed Hecke L-function associated with $\chi:$
\begin{equation}\label{fe}
\Lambda_F(\chi,s):=\left(D_FN_{F/{\mathbb{Q}}}(\mathfrak{m})\right)^{s/2}L_{F,\infty}(\chi,s)L_F(\chi,s),
\end{equation}
where $L_{F,\infty}(\chi,s)$ is the infinite part of $L_F(\chi,s),$ i.e.,
\begin{equation}\label{fe1}
L_{F,\infty}(\chi,s):=2^{-r_2s}\pi^{-ns/2}\Gamma(s/2)^a\Gamma((s+1)/2)^b\Gamma(s)^{r_2}.
\end{equation}
Then $\Lambda_F(\chi,s)$ is an entire function satisfying the functional equation:
\begin{equation}\label{4.3}
\Lambda_F(\chi,s)=W_{\chi}\Lambda_F(\overline{\chi},1-s),
\end{equation}
where $W_{\chi}\in\mathbb{C}$ is the root number with $\Big|W_{\chi}\Big|=1.$

Let $\Theta_{\chi}(x)$ be the inverse Mellin transform of $\Lambda_F(\chi,s).$ Then one can verify easily that \eqref{4.3} yields the functional equation of $\Theta_{\chi}:$
\begin{equation}\label{4.4}
\Theta_{\chi}(x)=x^{-1}W_{\chi}\Theta_{\overline{\chi}}(x^{-1}).
\end{equation}
Then by Mellin transform and \eqref{4.4} we have the integral representation of $\Lambda_F(\chi,s):$
\begin{align*}
\Lambda_F(\chi,s)=\int_{0}^{\infty}x^{s-1}\Theta_{\chi}(x)dx=\int_{1}^{\infty}x^{s-1}\Theta_{\chi}(x)dx+W_{\chi}\int_{1}^{\infty}x^{-s}\Theta_{\overline{\chi}}(x)dx.
\end{align*}
Since $\Gamma(s)$ is the Mellin transform of the function $f(x)=e^{-x},$ ($x>0$), the inverse Mellin transform of the $g_{\lambda,\mu}(s):=\Gamma(\lambda s+\mu),$ ($\forall$ $\lambda>0,$ $\mu\in \mathbb{R}$) is
$$
\left(\mathcal{M}^{-1}g_{\lambda,\mu}\right)(x)=\lambda^{-1}x^{\mu/\lambda}e^{-x^{1/{\lambda}}},\quad \lambda>0,\; \mu\in \mathbb{R}.
$$
Clearly, $\mathcal{M}^{-1}g_{\lambda,\mu}$ is positive. Since the parameters $a,$ $b$ and $n$ and nonnegative integers, so we can regard $\Gamma(s/2)^a\Gamma((s+1)/2)^b\Gamma(s)^{r_2}$ as a product of Gamma functions of the form $\Gamma(\lambda s),$ $\lambda>0.$ Let $\Theta_{\chi,\infty}$ denote the inverse Mellin transform of $2^{r_2s}\pi^{ns/2}L_{F,\infty}(\chi,s)=\Gamma(s/2)^a\Gamma((s+1)/2)^b\Gamma(s)^{r_2},$ then we have
$$
\Theta_{\chi,\infty}(x)=\left(\mathcal{M}^{-1,\ast a}g_{1/2,0}\ast\mathcal{M}^{-1,\ast b}g_{1/2,1/2}\ast\mathcal{M}^{-1,\ast{r_2}}g_{1,0}\right)(x),\quad\forall\; x>0,
$$
where for any $m\in\mathbb{N}_{+},$ $\mathcal{M}^{-1,\ast m}g$ denotes the $m-$fold convolution of $\mathcal{M}^{-1}g,$ the inverse Mellin transform of the function $g.$ Hence $\Theta_{\chi,\infty}(x)>0,$ $\forall$ $x>0.$

Note that by definition one has
\begin{align*}
\Theta_{\chi}(x)=\frac1{2\pi i}\int_{c-i\infty}^{c+i\infty}x^{-s}\Lambda_F(\chi,s)ds=\sum_{0\neq\mathfrak{a}\subset\mathcal{O}_F}\chi(\mathfrak{a})
\Theta_{\chi,\infty}\left(\frac{2^{r_2}\pi^{n/2}N_{F/{\mathbb{Q}}}(\mathfrak{a})}{\sqrt{D_FN_{F/{\mathbb{Q}}}(\mathfrak{m})}}\cdot x\right)
\end{align*}
By the definition of $a$ and $b$ we see $\Theta_{\overline{\chi},\infty}=\Theta_{\chi,\infty}.$ Since $\Theta_{\chi,\infty}$ is positive, we have $\Theta_{\overline{\chi}}=\overline{\Theta_{\chi}}.$ Thus for any $s>1,$ one obtains
\begin{align*}
\Big|\Lambda_F(\chi,1)\Big|&\leq\Big|\int_{1}^{\infty}\Theta_{\chi}(x)dx\Big|+\Big|\int_{1}^{\infty}x^{-1}\Theta_{\overline{\chi}}(x)dx\Big|\\
&\leq 2\Bigg|\int_{1}^{\infty}\sum_{0\neq\mathfrak{a}\subset\mathcal{O}_F}\chi(\mathfrak{a})
\Theta_{\chi,\infty}\left(\frac{2^{r_2}\pi^{n/2}N_{F/{\mathbb{Q}}}(\mathfrak{a})}{\sqrt{D_FN_{F/{\mathbb{Q}}}(\mathfrak{m})}}\cdot x\right)dx\Bigg|\\
&\leq2\sum_{0\neq\mathfrak{a}\subset\mathcal{O}_F}\int_{0}^{\infty}x^{s-1}
\Theta_{\chi,\infty}\left(\frac{2^{r_2}\pi^{n/2}N_{F/{\mathbb{Q}}}(\mathfrak{a})}{\sqrt{D_FN_{F/{\mathbb{Q}}}(\mathfrak{m})}}\cdot x\right)dx\\
&=2\left(D_FN_{F/{\mathbb{Q}}}(\mathfrak{m})\right)^{s/2}L_{F,\infty}(\chi,s)\zeta_{F}(s).
\end{align*}
From this inequality we obtains the upper bound for $L_F(\chi,1):$
\begin{equation}\label{4.5}
\Big|L_F(\chi,1)\Big|\leq2\left(D_FN_{F/{\mathbb{Q}}}(\mathfrak{m})\right)^{\frac{s-1}2}\frac{L_{F,\infty}(\chi,s)}{L_{F,\infty}(\chi,1)}\cdot\zeta(s)^n,
\quad\forall\; s>1.
\end{equation}
Let $H(s):=s^{a+r_2}(s-1)^n\frac{L_{F,\infty}(\chi,s)}{L_{F,\infty}(\chi,1)}\cdot\zeta(s)^n=\xi(s)^nG(s)^{b+r_2},$ where
\begin{align*}
\xi(s):=s(s-1)\pi^{-s/2}\Gamma(s/2)\zeta(s),\quad\text{and}\;  G(s):=\frac{\sqrt{\pi}\Gamma((s+1)/2)}{s\Gamma(s/2)}.
\end{align*}
Then by \eqref{4.5} we have $\Big|L_F(\chi,1)\Big|\leq{2\left(D_FN_{F/{\mathbb{Q}}}(\mathfrak{m})\right)^{\frac{s_0-1}2}}{(s_0-1)^n}^{-1}\cdot s_0^{-a-r_2}H(s_0),$
where $s_0:=1+2n[\log \left(D_FN_{F/{\mathbb{Q}}}(\mathfrak{m})\right)]^{-1}.$ Recall that we have the well known Hadamard decomposition for entire functions
\begin{align*}
\text{$\frac1{\Gamma(s)}=se^{\gamma s}\prod_{n=1}^{\infty}\left(1+\frac{s}{n}\right)e^{-s/n},$ and\; $\xi(s)=e^{Bs}\prod_{\rho}\left(1-\frac{s}{\rho}\right)e^{s/\rho},$}
\end{align*}
where $B\in\mathbb{C}$ is a constant and $\rho$ runs through nontrivial zeros of $\zeta(s).$ Then
$$
{G'(s)}/{G(s)}=\left(\log G(s)\right)'=\sum_{j=1}^{\infty}(-1)^{j}(j+s)^{-1}\leq0,\quad\forall\; s>0.
$$
This gives that $\left(\log G\right)''(s)=\sum_{j=1}^{\infty}(-1)^{j-1}(j+s)^{-2}>0$ when $s>0.$ Hence $\log G$ is convex when $s>0.$ Now we work out $B$ and thus see for every nontrivial zero $\rho=\sigma+i\tau,$ we have $|\tau|\geq6.$ In fact, by definition, $B=\xi(0)'/\xi(0).$ The functional equation $\xi(s)=\xi(1-s)$ gives
$\xi(s)'/\xi(s)=-\xi(1-s)'/\xi(1-s).$ Thus $B=-\xi(1)'/\xi(1).$ Therefore,
$$
B=\frac12\log\pi-\frac12\frac{\Gamma'}{\Gamma}\left(\frac32\right)-\lim_{s\mapsto1^{+}}\left(\frac{\zeta'}{\zeta}(s)+\frac1{s-1}\right)
=-1-\frac{\gamma}{2}+\frac12\log{4\pi}.
$$
On the other hand, by $B=-\xi(1)'/\xi(1)$ and symmetry of the nontrivial zeros, 
$$
B=-\frac12\sum_{\rho}\left(\frac1{1-\rho}+\frac1{\rho}\right)=-\sum_{\rho=\sigma+i\tau}\mathfrak{Re}\frac1{\rho}
=-\sum_{\rho=\sigma+i\tau}\frac{2\sigma}{\sigma^2+\tau^2}.
$$
Then for any $\rho=\sigma+i\tau$ with $1/2\leq\rho\leq1,$ one has $-B\geq2\sigma\cdot(\sigma^2+\tau^2)^{-1},$ which gives the lower bound $|\tau|\geq\sqrt{\frac{2\sigma}{-B}-\sigma^2}\geq\sqrt{\frac{1}{-B}-\frac14}\geq6.$
Thus the function
$$
h(s):=\sum_{\rho=\sigma+i\tau}\frac1{s-\rho}=\sum_{\substack{\rho=\sigma+i\tau\\ \tau\geq0}}\frac{s-\sigma}{(s-\sigma)^2+\tau^2},\quad 1\leq s\leq6,
$$
is increasing. So $(\log\xi(s))'=\xi(s)'/\xi(s)=B+h(1)+h(s)$ is increasing when $1\leq s\leq 6.$ Then $\log\xi(s)$ is convex when $1\leq s\leq6.$
Therefore we have $\log\left(s^{-a-r_2}H(s)\right)$ is convex when $1\leq s\leq6.$

By \cite{Zi81} there exists an integral ideal $\mathfrak{a}$ such that $N_{F/{\mathbb{Q}}}(\mathfrak{a})\leq M(r_1,r_2)\sqrt{D_F}.$ Then clearly
$$
\log \left(D_FN_{F/{\mathbb{Q}}}(\mathfrak{m})\right)\geq\log D_F\geq M(r_1, r_2)^{-1}.
$$
Recall that $M(r_1, r_2)\leq M(n):=\frac{4^{r_2}n!}{\pi^{r_2}n^n}.$ Hence
\begin{align*}
s_0&=1+2n[\log \left(D_FN_{F/{\mathbb{Q}}}(\mathfrak{m})\right)]^{-1}\leq1+2nM(r_1, r_2) \leq1+{2n}\cdot\frac{4^{n}n!}{\pi^{n}n^n}\leq6.
\end{align*}
By convexity we have $H(s_0)\leq\max\{H(1), H(6)\}=1.$

Let $\chi\in\widehat{Cl(F)}$ is a nontrivial Hilbert character. One has that
$$
\text{$Cl(F)\simeq I_F/{I_F^{\infty}F^{\times}},$ where $I_F^{\infty}:=\prod_{\mathfrak{p}\mid\infty}F^{\times}_{\mathfrak{p}}\times\prod_{\mathfrak{p}\nmid\infty}\mathcal{O}^{\times}_{F, \mathfrak{p}}.$}
$$
So in this case $b=0,$ and $a=r_1,$ and $\mathfrak{m}=\mathcal{O}_F.$ Then we have the completed L-function \eqref{fe}, where $N_{F/{\mathbb{Q}}}(\mathfrak{m})=1$ and \eqref{fe1} becomes
\begin{align*}
L_{F,\infty}(\chi,s):=2^{-r_2s}\pi^{-ns/2}\Gamma(s/2)^{r_1}\Gamma(s)^{r_2}.
\end{align*}
As before, noting that $\Theta_{\overline{\chi}}=\overline{\Theta_{\chi}},$ Mellin transform and functional equations imply
\begin{equation}\label{fe3}
\Big|\Lambda_F(\chi,1)\Big|\leq\int_{1}^{\infty}|\Theta_{\chi}(x)|(1+x^{-1})dx\leq \int_{1}^{\infty}|\Theta(x)|(1+x^{-1})dx,
\end{equation}
where let $\Theta_{\infty}$ is the inverse Mellin transform of $\Gamma(s/2)^{r_1}\Gamma(s)^{r_2},$ and
$$
\Theta(x):=\sum_{0\neq\mathfrak{a}\subset\mathcal{O}_F}
\Theta_{\infty}\left(\frac{2^{r_2}\pi^{n/2}N_{F/{\mathbb{Q}}}(\mathfrak{a})}{\sqrt{D_F}}\cdot x\right).
$$
Let $c>10,$ and $\Lambda_F(s):=D_F^{s/2}2^{-r_2s}\pi^{-ns/2}\Gamma(s/2)^{r_1}\Gamma(s)^{r_2}\zeta_F(s).$ Then by \eqref{fe3}, 
\begin{align*}
\Big|\Lambda_F(\chi,1)\Big|&\leq\int_{1}^{\infty}\left(\frac1{2\pi i}\int_{c-i\infty}^{c+\infty}\Lambda_F(s)x^{-s}ds\right)(1+x^{-1})dx
=\frac1{2\pi i}\int_{c-i\infty}^{c+\infty}\widetilde{\Lambda}_F(s)ds,
\end{align*}
where $\widetilde{\Lambda}_F(s):=\Lambda_F(s)\cdot\left(\frac1s+\frac1{s-1}\right).$ Denote the right hand side by $I_F.$ The functional equation $\Lambda_F(s)=\Lambda_F(1-s)$ gives us that $\widetilde{\Lambda}_F(s)=-\widetilde{\Lambda}_F(1-s).$

Recall that combining the elementary bound and functional equation of $\zeta_F(s)$ and by the Phragm\'{e}n-Linderl\"{o}f theorem we have the fact that
$$
\zeta_F(\sigma+it)\ll |t|^{\frac{1-\sigma}2}\log |t|,\quad 0\leq \sigma\leq 1,\; |t|\geq2.
$$
By functional we have $\zeta_F(\sigma+it)\ll |t|^{1/2-\sigma}\log |t|,\quad \sigma\leq 0,\; |t|\geq2.$ Note that in the area $\mathcal{S}:=\{s=\sigma+it:\; 1-c\leq\sigma\leq c, |t|\geq1\}$ we have uniformly that
$$
\Gamma(s)=\sqrt{2\pi}e^{-\frac{\pi}{2}|t|}e^{it(\log|t|-1)}e^{\frac{i\pi t}{2|t|}(\sigma-1/2)}\left(1+O_c\left(|t|^{-1}\right)\right).
$$
So $\widetilde{\Lambda}_F(s)$ decays exponentially in $\mathcal{S}$ as $|t|\mapsto\infty.$ Thus shifting the contour one gets
\begin{align*}
I_F&=Res_{s=1}\widetilde{\Lambda}_F(s)+Res_{s=1}\widetilde{\Lambda}_F(s)+
\frac1{2\pi i}\int_{1-c-i\infty}^{1-c+i\infty}\widetilde{\Lambda}_F(s)ds\\
&=2Res_{s=1}\widetilde{\Lambda}_F(s)-\frac1{2\pi i}\int_{c-i\infty}^{c+\infty}\widetilde{\Lambda}_F(s)ds.
\end{align*}
Hence $I_F=Res_{s=1}\widetilde{\Lambda}_F(s)=\Lambda_{F,1}+\Lambda_{F,2},$ where
$$
\text{$\Lambda_{F,1}:=\lim_{s\rightarrow1}(s-1)\cdot\Lambda_{F}(s),$ and $\Lambda_{F,2}:=\lim_{s\rightarrow1}(s-1)\cdot\Lambda_{F}(s)'.$}
$$
 By the Laurent expansion of $\zeta_F(s)$ at $s=1$ we have
$\lim_{s\rightarrow1}\left(\frac1{s-1}+\frac{\zeta_F'}{\zeta_F}(s)\right)=\rho_F^{-1}\gamma_F.$ Hence ${\Lambda_{F,2}}/{\Lambda_{F,1}}=\left(\log\Lambda_F(s)\right)'\mid_{s=1}$ is equal to
\begin{align*}
&\frac12\log D_F-r_2\log2-\frac{n}2\log\pi+\frac{r_1}{2}\cdot\frac{\Gamma'}{\Gamma}\left(\frac{1}2\right)+r_2\frac{\Gamma'}{\Gamma}\left(1\right)
+\rho_F^{-1}\gamma_F\\
=&\frac12\log D_F-\frac{n}2\log\pi+\frac{\gamma+2\log2}2\cdot r_1+(\gamma+\log2)r_2
+\rho_F^{-1}\gamma_F.
\end{align*}
Then by the inequality $|\Lambda_F(\chi,1)|\leq I_F$ we obtain \eqref{3.5}.
\end{proof}
Now we move on to handle the Euler-Kronecker constant $\gamma_F.$ Clearly we need an upper bound for it. The known result on upper bounds for $\gamma_F$ is essentially $2\log\log\sqrt{D_F},$ which is established under GRH (ref. \cite{Ih06}). To prepare for the proof of Theorem \ref{main}, we give an elementary unconditional effective upper bound for $\gamma_F.$ 
\begin{lemma}\label{28''}
	Let notation be as before. Then there is an absolute constant $c>0$ such that 
	\begin{equation}\label{EK}
	-\frac12\log D_F-\frac{\gamma+2\log2-\log\pi}2n-1\leq\gamma_F^*\leq c\log D_F.
	\end{equation}
	\end{lemma}
\begin{remark}
	Note that the main term of this lower bound in \eqref{EK} is $-\frac12 \log D_F,$ which is slightly better than the general result (i.e. lower bound of main term $-\log D_F$) given in \cite{Ih06}. On the other hand, under GRH, one has $\gamma_F^*\ll \log\log D_F$ according to main theorems in \cite{Ih06}.
	\end{remark}
\begin{proof}
	The lower bound for $\gamma_F$ can be deduced simply from  \eqref{3.5}. We thus will focus on the upper bound here. For $s=\sigma+it\in\mathbb{C}$ with $\frac12\leq \sigma\leq1,$ one has
	\begin{equation}\label{5.30}
	(s-1)\zeta_F(s)\ll \rho_F|sD_F|^{(1-\sigma)/2}.
	\end{equation}
	Note \eqref{5.30} is essentially Theorem 5.31 in \cite{IK04} without $|sD_F|^{\varepsilon}.$ We drop the $\varepsilon$-factor by using a subconvexity bound for $\zeta_F$ at $\sigma=\Re(s)=1/2$ as an endpoint rather than using the bound for $\zeta_F(it)$ via the functional equation before applying Phragm\'{e}n-Linderl\"{o}f theorem.
	In fact, a subconvexity result (e.g. ref. \cite{Ve10}) shows that there exists some $\delta$ (e.g. one can take $\delta=1/200$) such that $\zeta_F(s)\ll |sD_F|^{1/2-\delta},$ where $s=1/2+it.$
	Then Phragm\'{e}n-Linderl\"{o}f theorem implies that 
	\begin{equation}\label{sub}
	(s-1)\zeta_F(s)\ll |sD_F|^{(1/2-\delta)(1-\sigma)},\quad \text{where $s=\sigma+it,$ $1/2\leq\sigma<1.$}
	\end{equation}
    Hence \eqref{5.30} comes from \eqref{sub} and Theorem 5.31 in \cite{IK04}. 
    
    Let $\mathcal{C}$ be the circle centered at $s=1$ with radius $r=1-(2\log D_F)^{-1}.$ Since $\zeta_F$ is a meromorphic function with a simple pole at $s=1$, then by Cauchy's theorem we have
    \begin{align*}
    \gamma_F=\frac{1}{2\pi i}\oint_{\mathcal{C}}\frac{\zeta_F(s)}{s-1}ds\ll \oint_{\mathcal{C}}\frac{\rho_F}{|s-1|^2}|ds|\ll \rho_F\log D_F.
    \end{align*} 
    Then the proof follows.
	\end{proof}
With these preparation we can eventually give a proof of our main theorems.
\subsection{Proof of Theorem \ref{main}}\label{sec4.2}
\begin{proof}
As before, let $K/F$ be a CM extension and $[F:\mathbb{Q}]=n.$ Note that for every $\chi\in\widehat{Cl(K)},$ the conductor of $\chi$ is $\mathcal{O}_K.$ So we have, by Lemma \ref{average}, that
 \begin{align*}
\rho_F^{-1}\mathcal{L}_{F}&\leq 1+\frac{\gamma+2\log2-\log\pi}2n+\frac12\log D_F
+\gamma_F^*.
\end{align*}
Then Lemma \ref{28''} implies that $\rho_F^{-1}\mathcal{L}_{F}\leq c_1' \log D_F$ for some absolute constant $c_1'>0,$ since a classical lower bound for $D_F$ implies that $\log D_F\gg n.$ 

Likewise, one has $\mathcal{L}_F^*=4\rho_F^{-1}\mathcal{L}_{F}+\log D_F+(4\log2-\pi)n+\sqrt{7n}+4\leq c'_2\log D_F$ for some positive absolute constant $c_2'.$  Then \eqref{0.1} follows from Proposition \ref{fin} and thus \eqref{0.01} follows from \eqref{0.1} and elementary computations of $M(n,0)$ and $M(0,n).$
\end{proof}
\subsection{Proof of Theorem \ref{c}}\label{sec4.3}
Substitute orthogonality into Proposition \ref{11} one has 
\begin{lemma}\label{x}
Let notations be as above, then we have
\begin{equation}\label{ll}
L_K\left(\chi,\frac12\right)=\frac{2^{\frac{n}2}\rho_F\sqrt{D_F}}{2D_K^{\frac14}[\mathcal{O}^{\times}_K:\mathcal{O}^{\times}_F]}\sum_{[\mathfrak{a}^{-1}]\in Cl(K)}\overline{\chi}([\mathfrak{a}])\sqrt{N(\mathfrak{f}_{\mathfrak{a}})}
E'\left(z_{\mathfrak{a}},\frac12;\mathfrak{f}^{-1}_{\mathfrak{a}},\mathfrak{f}_{\mathfrak{a}}\right),
\end{equation}
and
\begin{align*}
\frac1{h_K}\sum_{\chi\in \widehat{Cl(K)}}\Big|L_K(\chi,1/2)\Big|^2
&=\frac{2^{n-2}D_F}{\sqrt{D_K}}\cdot\frac{\rho_F^2}{[\mathcal{O}^{\times}_K:\mathcal{O}^{\times}_F]^2}\\
&\quad\times\sum_{[\mathfrak{a}^{-1}]\in Cl(K)}N_{F/\mathbb{Q}}(\mathfrak{f}_{\mathfrak{a}})
\Bigg|E'\left(z_{\mathfrak{a}},\frac12;\mathfrak{f}^{-1}_{\mathfrak{a}},\mathfrak{f}_{\mathfrak{a}}\right)\Bigg|^2.
\end{align*}
\end{lemma}
After inserting the Fourier expansion Lemma \ref{Der} and the CM type norm formula \eqref{cm2} into \eqref{ll}, one gets the following generalization of \eqref{1.1}:
\begin{prop}\label{main3}
	Let notation be as before. Let $\mathfrak{a}$ be a fractional ideal of $K.$ Then one has
	\begin{equation}\label{main3'}
	\frac1{h_K}\sum_{\chi\in \widehat{Cl(K)}}\chi(\mathfrak{a})L_K(\chi,1/2)
	=\mathcal{N}_{\mathfrak{q}}(\mathfrak{a})
	\cdot\Big\{\log N_{\Phi}(y_{\mathfrak{a}})+\mathcal{E}_1(F,K;\mathfrak{a})\Big\},
	\end{equation}
	where 
	\begin{equation}\label{N}
	\mathcal{N}_{\mathfrak{q}}(\mathfrak{a})=\frac{\rho_F\cdot N_{F/{\mathbb{Q}}}(\mathfrak{q})}{[\mathcal{O}^{\times}_K:\mathcal{O}^{\times}_F]\cdot h_F}\cdot \sqrt{\frac{N_{K/{\mathbb{Q}}}(c_{\mathfrak{a}})}
		{N_{K/{\mathbb{Q}}}(\mathfrak{a})}},
	\end{equation}
	here $c_{\mathfrak{a}}$ is given by \eqref{cm2}, and $\mathfrak{f}_{\mathfrak{a}}$ is determined by $\mathfrak{a}$ according to \eqref{cm0}. The error term $\mathcal{E}_1(F,K;\mathfrak{a})$ above satisfies that
	\begin{equation}\label{mainerror}
	\mathcal{E}_1(F,K;\mathfrak{a})\ll \log D_F+\frac{h_F^2D_F^{1/4}N(\mathfrak{f}_{\mathfrak{a}})}{\rho_Fh^{+}_F}N_{\Phi}(y_{\mathfrak{a}})^{-3/2}.
	\end{equation}
	where the implied constant in is absolute.
\end{prop}
Take $\mathfrak{a}$ to be trivial in \eqref{main3'}, combining with \eqref{cm2}, \eqref{N} and \eqref{mainerror}, we have 
\begin{equation}\label{xx}
\frac1{h_K}\sum_{\chi\in \widehat{Cl(K)}}L_K(\chi,1/2)\gg_{F} \frac{\log D_K}{[\mathcal{O}^{\times}_K:\mathcal{O}^{\times}_F]}.
\end{equation}
On the other hand, substituting Lemma \ref{Der} into the second formula in Lemma \ref{x} leads to 
\begin{align*}
\frac1{h_K}\sum_{\chi\in \widehat{Cl(K)}}\Big|L_K(\chi,1/2)\Big|^2
&\ll\frac{2^{n-2}D_F^2}{\sqrt{D_K}}\cdot\frac{\rho_F^2}{[\mathcal{O}^{\times}_K:\mathcal{O}^{\times}_F]^2}\sum_{[\mathfrak{a}^{-1}]}
\Bigg|E'\left(z_{\mathfrak{a}},\frac12;\mathfrak{f}^{-1}_{\mathfrak{a}},\mathfrak{f}_{\mathfrak{a}}\right)\Bigg|^2\\
&\ll_F\frac{(\log D_K)^2}{[\mathcal{O}^{\times}_K:\mathcal{O}^{\times}_F]^2}\sum_{N(\mathfrak{a})\ll \sqrt{D_K}}\frac{1}{N(\mathfrak{a})}.
\end{align*}
Then a standard estimate on $\sum_{N(\mathfrak{a})\ll\sqrt{D_K}}\frac{1}{N(\mathfrak{a})}$ implies that 
\begin{equation}\label{zz}
\frac1{h_K}\sum_{\chi\in \widehat{Cl(K)}}\Big|L_K(\chi,1/2)\Big|^2\ll_{F}\frac{(\log D_K)^2}{[\mathcal{O}^{\times}_K:\mathcal{O}^{\times}_F]^2}\sum_{n\ll\sqrt{D_K}}\frac{d(n)}{n}\ll_F\frac{(\log D_K)^3}{[\mathcal{O}^{\times}_K:\mathcal{O}^{\times}_F]^2},
\end{equation}
where $d(n)$ is the divisor function, and the implied constants are effective. 
\begin{proof}[Proof of Theorem \ref{c}]
By \eqref{xx}, \eqref{zz} and Cauchy inequality, we have 
\begin{align*}
\frac{(h_K\log D_K)^2}{[\mathcal{O}^{\times}_K:\mathcal{O}^{\times}_F]^2}&\ll_F\Big|\sum_{\chi\in \widehat{Cl(K)}}L_K(\chi,1/2)\Big|^2\\
&\ll_F \#\Big\{\chi\in\widehat{Cl(K)}:\; L_K\left(\chi,1/2\right)\neq0\Big\}\cdot\sum_{\chi\in \widehat{Cl(K)}}\Big|L_K(\chi,1/2)\Big|^2\\
&\ll_{F,\epsilon} \#\Big\{\chi\in\widehat{Cl(K)}:\; L_K\left(\chi,1/2\right)\neq0\Big\}\cdot\frac{h_K(\log D_K)^3}{[\mathcal{O}^{\times}_K:\mathcal{O}^{\times}_F]^2}.
\end{align*} 
 Now the $k=0$ case of Theorem \ref{c} follows. The $k=1$ case then comes from the $k=0$ case and logarithmic derivative of the functional equation of $L_K(\chi,s).$
\end{proof}

\end{document}